\documentclass[11pt]{amsart}
\usepackage{enumerate,amsmath,amssymb, mathrsfs}

\usepackage{latexsym, amssymb,enumerate}
\usepackage{appendix}

\newcommand{\be}{\begin{equation}}
\newcommand{\ee}{\end{equation}}
\newcommand{\beq}{\begin{eqnarray}}
\newcommand{\eeq}{\end{eqnarray}}
\newcommand{\cM}{\mathcal{M}}

\def\H{{\mathbb H}}

\def\R{{\mathfrak R}}

\def\bv{\bar\nabla}

\newtheorem{prop}{Proposition}[section]
\newtheorem{theo}[prop]{Theorem}
\newtheorem{lemm}[prop]{Lemma}

\newtheorem{rema}[prop]{Remark}
\newtheorem{exam}[prop]{Example}
\newtheorem{defi}[prop]{Definition}
\newtheorem{conj}[prop]{Conjecture}

\def\begeq{\begin{equation}}
\def\endeq{\end{equation}}
\def\p{\partial}

\def\R{\mathbb R}

\def\tr{{\rm tr}}

\def\d{\delta}
\def\a{\alpha}

\def\s{\sigma}

\def \ds{\displaystyle}
\def \vs{\vspace*{0.1cm}}

\def\l{\lambda}
\def\Hess{{\rm Hess\,}}

\def\S{\mathbb  {S}}
\def\div{{\rm div\,}}

\def\odot{\setbox0=\hbox{$\bigcirc$}\relax \mathbin {\hbox
to0pt{\raise.5pt\hbox to\wd0{\hfil $\wedge$\hfil}\hss}\box0 }}

\numberwithin{equation} {section}

\def\tilde{\widetilde}

\begin{document}

\title[Hyperbolic Gauss-Bonnet-Chern mass] {The GBC mass for asymptotically hyperbolic manifolds}

\author{Yuxin Ge}
\address{Laboratoire d'Analyse et de Math\'ematiques Appliqu\'ees,
CNRS UMR 8050,
D\'epartement de Math\'ematiques,
Universit\'e Paris Est-Cr\'eteil Val de Marne, \\61 avenue du G\'en\'eral de Gaulle,
94010 Cr\'eteil Cedex, France}
\email{ge@u-pec.fr}
\author{Guofang Wang}
\address{ Albert-Ludwigs-Universit\"at Freiburg,
Mathematisches Institut
Eckerstr. 1
D-79104 Freiburg}
\email{guofang.wang@math.uni-freiburg.de}

\author{Jie Wu}
\address{School of Mathematical Sciences, University of Science and Technology
of China Hefei 230026, P. R. China
\and
 Albert-Ludwigs-Universit\"at Freiburg,
Mathematisches Institut
Eckerstr. 1
D-79104 Freiburg
}
\email{jie.wu@math.uni-freiburg.de}

\begin{abstract} The paper consists of  two parts. In the first part,
by using the Gauss-Bonnet curvature, which is a natural generalization of the scalar curvature,
 we introduce a higher order mass,
the Gauss-Bonnet-Chern mass $m^{\H}_k$, for asymptotically hyperbolic manifolds
and show that it is a geometric invariant. Moreover, we prove a
positive mass theorem for this new  mass for asymptotically hyperbolic graphs and establish a relationship between
the corresponding Penrose type inequality for this  mass and weighted Alexandrov-Fenchel inequalities in the hyperbolic space $\H^n$.
In the second part,
 we  establish  these weighted Alexandrov-Fenchel inequalities in  $\H^n$ for any horospherical convex hypersurface $\Sigma$
\[ \int_{\Sigma} V \s_{2k-1}  d\mu \ge C_{n-1}^{2k-1} \omega_{n-1}{\left(\left(\frac{|\Sigma|}{\omega_{n-1}}\right)^{\frac{n}{k(n-1)}}+\left(\frac{|\Sigma|}{\omega_{n-1}}\right)^{\frac{n-2k}{k(n-1)}} \right)}^{k},\]
where $\s_{j}$ is the $j$-th  mean curvature of $\Sigma \subset\H^n$, $V=\cosh r$ and $|\Sigma|$ is the area of $\Sigma$.
 As an application, we obtain an optimal  Penrose type inequality
for the new mass defined in the first part
\[m_{k}^{\H} \ge
\frac{1}{2^k}{\left(\left(\frac{|\Sigma|}{\omega_{n-1}}\right)^{\frac{n}{k(n-1)}}+\left(\frac{|\Sigma|}{\omega_{n-1}}\right)^{\frac{n-2k}{k(n-1)}} \right)}^{k},\]
for asymptotically hyperbolic graphs 
 with a horizon type boundary $\Sigma$, provided that a dominant energy condition $\tilde L_k\ge0$ holds.
 Both inequalities are optimal.
\end{abstract}

\thanks{This project is partly supported by SFB/TR71
``Geometric partial differential equations''  of DFG}

\subjclass[2000]{53C21, (83C05, 83C30)}

\keywords{Gauss-Bonnet-Chern mass,  Gauss-Bonnet curvature, asymptotically hyperbolic manifold, positive mass theorem,
 Penrose inequality, Alexandrov-Fenchel inequality}

\maketitle

\tableofcontents

\part{The GBC mass for asymptotically hyperbolic manifolds}

\section{Introduction}

The study of the scalar curvature plays an important role in differential geometry. There are many beautiful results on manifolds
with positive or non-negative  scalar curvature.
One of them is the Riemannian positive mass theorem (PMT): {\it Any asymptotically flat Riemannian manifold $\cM^n$ with a suitable decay order
and with nonnegative scalar curvature
has the nonnegative ADM mass. Moreover, equality holds if and only if the manifold $\cM^n$  is isometric to the Euclidean space $\R^n$ with the standard metric.}   This theorem was first proved by Schoen and Yau \cite{SY} for manifolds of dimension $n\le 7$
by using a minimal hypersurface argument,
 and later for spin manifolds by Witten \cite{Wit}. See also \cite{PT}.
 For locally conformally flat manifolds the proof was given in \cite{Schoen2}
using the developing map.
Very recently, for the special case of asymptotically flat Riemannian manifolds $\cM^n$ which can be represented by  graphs over $\R^{n}$, Lam \cite{Lam} gave a proof by observing that the scalar curvature of these asymptotically flat graphs can be expressed exactly as a divergence term.
For general  higher dimensional manifolds, the proof of the  positive mass theorem was announced by Schoen \cite{Schoen_talk} and
Lohkamp \cite{Loh} by an argument extending the minimal hypersurface argument of Schoen and Yau. There are many generalizations and applications
of the positive mass theorem.
See \cite{Bar, Dai, Shi_Tam, Miao} for instance.

A refinement of the PMT  is  the Riemannian Penrose inequality
\beq\label{eq01}
m_1=m_{ADM} \geq \frac 12\left(\frac{|\Sigma|}{\omega_{n-1}}\right)^{\frac{n-2}{n-1}},
\eeq
 where $m_{ADM} $ is the ADM mass of the  asymptotically flat Riemannian manifold
with a horizon $\Sigma$  and $|\Sigma|$ denotes the area of $\Sigma$. (\ref{eq01}) was
 proved by Huisken-Ilmanen \cite{HI}
 and Bray \cite{BP}
 for $n=3$.
Later, Bray and Lee \cite{BL} generalized Bray's proof to the case  $n\le 7$. See also the excellent surveys  of Bray \cite{Bray} and Mars \cite{Mars}.
Recently, Lam \cite{Lam} gave an elegant proof of (\ref{eq01}) in all dimensions for an   asymptotically flat
manifold which is a graph. His idea  was extended by Huang-Wu in \cite{HuangWu}  and de Lima-Gir\~ao in
\cite{dLG1,dLG2} to submanifolds of $\R^{n+1}$ and  of general ambient spaces.

The ADM mass, together with the positive mass theorem, was generalized to asymptotically hyperbolic manifolds in \cite {{AnderssonCaiGalloway}, BQ, CN,CH,Wang, ZhangX}.
For this asymptotically hyperbolic mass, the corresponding  Penrose  conjecture is: {\it For asymptotically hyperbolic manifold $(\cM^n,g)$ with an outermost horizon $\Sigma$, its mass satisfies
\begin{equation} \label{1Penrose}
m_1^{\H}=m^{\H}\geq
\frac{1}{2}\left\{
\left( \frac{|\Sigma|}{\omega_{n-1}} \right)^{\frac{n-2}{n-1}}
+ \left(\frac{|\Sigma|}{\omega_{n-1}}\right)^{\frac{n}{n-1}}
\right\},
\end{equation}
provided that the dominant energy condition
\beq \label{D_con1}
R_g \geq -n(n-1),
\eeq
holds.}  Here $R_g$ denotes the scalar curvature of $g$. Neves \cite{Neves} showed that the powerful inverse mean curvature flow of Huisken-Ilmanen
\cite{HI} alone could not be used to prove \eqref{1Penrose}. (However, very recently Lee and Neves have a new development in this direction \cite{LeeNeves}.)
 Recently, motivated by the work of Lam \cite{Lam}, Dahl-Gicquaud-Sakovich \cite{DGS}
and de Lima and Gir\~ao \cite{dLG} proved the Penrose inequality \eqref{1Penrose} for asymptotically hyperbolic graphs over $\H^n$.
More precisely, Dahl-Gicquaud-Sakovich \cite{DGS} proved for asymptotically hyperbolic graphs generated by a function $f:\H^n\backslash \Omega\to \R$,
the  hyperbolic mass satisfies a crucial estimate in terms of  a weighted mean curvature integral
\begin{equation} \label{DGS1}
m_1^{\H} \ge \frac 1{2(n-1)\omega_{n-1}}  \int_{\Sigma} V H d\mu,
\end{equation}
where $\Sigma=\partial \Omega$, $H$ is the mean curvature of $\Sigma$  induced by the hyperbolic metric $b$, $\omega_{n-1}$ is the area of the unit sphere, $d\mu$ is the volume element induced by $b$, $V=\cosh r$ and  $r$ is the hyperbolic distance from an arbitrary fixed point on $\H^n$. They also gave an almost sharp estimate for $\int_{\Sigma} V H d\mu$,
and hence an almost sharp estimate for the hyperbolic mass. The sharp estimate for $\int_{\Sigma} V H d\mu$ is a weighted  hyperbolic Minkowski inequality, or a weighted hyperbolic
Alexandrov-Fenchel inequality
\begin{equation} \label{eq_dLG}
\ds  \int_{\Sigma}  V H d\mu \ge (n-1)  \omega_{n-1}\left\{\left(\frac{|\Sigma|}{\omega_{n-1}}\right)^{\frac{n-2}{n-1}}+\left(\frac{|\Sigma|}{\omega_{n-1}}\right)^{\frac{n}{n-1}}\right\},
\end{equation}
if $\Sigma$ is star-shaped and mean-convex (i.e. $H>0$). This result  was proved by de Lima and Gir\~ao \cite{dLG}.
A closely related Minkowski type inequality was proved by Brendle-Hung-Wang \cite{BHW}, which
is not only true for the hyperbolic space, but also for anti-de Sitter Schwarzschild manifolds. For the proof of  \eqref{eq_dLG},
a Heintze-Karcher type inequality proved by Brendle in \cite{Brendle} and an inverse curvature flow studied by Gerhardt \cite{Gerhardt} play an important role.

Recently motivated by the Gauss-Bonnet gravity \cite{DT1, DT2}, we have introduced  the Gauss-Bonnet-Chern mass $m_{GBC}$ for asymptotically flat manifolds by using the
following Gauss-Bonnet  curvature
 \begin{equation}\label{Lk2}
L_k:=\frac{1}{2^k}\d^{i_1i_2\cdots i_{2k-1}i_{2k}}
_{j_1j_2\cdots j_{2k-1}j_{2k}}{R_{i_1i_2}}^{j_1j_2}\cdots
{R_{i_{2k-1}i_{2k}}}^{j_{2k-1}j_{2k}}.
\end{equation}
Here $\d^{i_1i_2\cdots i_{2k-1}i_{2k}}
_{j_1j_2\cdots j_{2k-1}j_{2k}}$ is the generalized delta defined by (\ref{generaldelta}) below and ${R_{ij}}^{sl}$ is the Riemannian curvature tensor. One can check that
  $L_1$ is just the scalar curvature $R$. When $k=2$, it is the (second) Gauss-Bonnet curvature
\[L_2 = \|Rm\|^2-4\|Ric\|^2+R^2,\]
which appeared at the first time in the paper of Lanczos \cite{Lan} in 1938. For general $k$, it is just the Euler integrand in the
Gauss-Bonnet-Chern theorem \cite{Chern1, Chern2} if $n=2k$ and
 is therefore called the dimensional continued Euler density
in physics if $k<n/2$. A systematic study of $L_k$ was first given by Lovelock \cite{Lovelock}. See also \cite{Pat}, \cite{Willa} and \cite{Labbi}.
The Gauss-Bonnet-Chern mass $m_{GBC}$ for the asymptotically flat manifolds is defined in \cite{GWW} by
\begin{equation}\label{GBC}
m_k=m_{GBC}=c(n,k)
\lim_{r\to\infty}\int_{S_r}P_{(k)}^{ijls}\partial_s g_{jl} \nu_{i}d\mu,\end{equation}
with
\begin{equation}\label{eq_const}
c(n,k)=\frac{(n-2k)!} {2^{k-1}(n-1)!\;\omega_{n-1}},\end{equation}
where $\omega_{n-1}$ is the volume of $(n-1)$-dimensional standard unit sphere and  $S_r$ is the Euclidean coordinate sphere, $d\mu$ is the volume element on $S_r$ induced by the Euclidean metric and  $\nu$ is the outward unit normal to $S_r$ in $\mathbb{R}^n$.
Here the $(0,4)$-tensor $P_{(k)}$ is defined in (\ref{Pk}) below.
This $(0,4)$-tensor has a crucial property that it is divergence-free. See Section 2 below.
For a similar definition see also \cite{LN}.
In \cite{GWW} and \cite{GWW3}, we  prove a positive mass theorem in the case that $\cM$ is an asymptotically flat graph over $\R^n$
or $\cM$ is  conformal  to $\R^n$ respectively. 
For our mass $m_{GBC}$, a corresponding Penrose conjecture was proposed in \cite{GWW}
\begin{equation} \label{PenroseE}
m_k=m_{GBC} \geq \frac{1}{2^k}\left(\frac{|\Sigma|}{\omega_{n-1}}\right)^{\frac{n-2k}{n-1}}.
\end{equation}
Moreover we proved in \cite{GWW} that this conjecture is true for asymptotically flat graphs over $\R^n\backslash \Omega$
by using classical Alexandrov-Fenchel inequalities, if
$\Sigma=\partial \Omega$ is convex. For  the classical Alexandrov-Fenchel inequalities see
excellent books  \cite{BuragoZalgaller, Santos, Schneider}.
The Alexandrov-Fenchel inequalities  also hold for certain classes of non-convex hypersurfaces.
See for example \cite{ChangWang,GuanLi3,Huisken}.

In this paper, motivated by our previous work, by using the Gauss-Bonnet curvature
we  introduce a higher order mass for asymptotically hyperbolic  manifolds, which is a generalization of the
mass introduced  by Wang \cite{Wang} and Cru\'sciel-Herzlich \cite{CH}.
See also \cite{CN,Herzlich,M, ZhangX}.
However,  if we use directly the Gauss-Bonnet curvature $L_k$, we  can only obtain a mass proportional to the the usual hyperbolic mass, rather than a new one. This phenomenon is different from the Euclidean case. 
In order to define a higher order mass for asymptotically hyperbolic manifolds, the crucial observation is a slight modification of the Gauss-Bonnet curvature. More precisely, on a Riemannian manifold $(\cM^n,g)$, we consider a modified Riemann curvature tensor
\be\label{Riem}
\widetilde  {\rm Riem}_{ijls}(g)=\tilde  R_{ijls}(g):=R_{ijls}(g)+g_{il}g_{js}-g_{is}g_{jl}.\ee
Clearly, the modified Riemann curvature tensor $\widetilde  {\rm Riem}$ has the same symmetry as the ordinary
Riemann curvature tensor $Riem$ and also satisfies the first and second Bianchi identities.
We define a new Gauss-Bonnet curvature with respect to this tensor $\widetilde {Riem}$
\begin{equation}\label{tildeLk_intro}
\tilde L_k:=\frac{1}{2^k}\d^{i_1i_2\cdots i_{2k-1}i_{2k}}
_{j_1j_2\cdots j_{2k-1}j_{2k}} {\tilde R_{i_1i_2}}^{\quad\, j_1 j_2}\cdots
{\tilde R_{i_{2k-1}i_{2k}}}^{\qquad\;\,j_{2k-1}j_{2k}}=\tilde R_{ijls}{\tilde P_{(k)}}^{ijls},
\end{equation}
where
\begin{equation}
{{\tilde P}_{(k)}}^{ijls}:=\frac{1}{2^k}\d^{i_1i_2\cdots i_{2k-3}i_{2k-2}ij}
_{j_1j_2\cdots  j_{2k-3}j_{2k-2} j_{2k-1}j_{2k}}{\tilde R_{i_1i_2}}^{\quad\, j_1 j_2}\cdots
{\tilde R_{i_{2k-3}i_{2k-2}}}^{\qquad\quad j_{2k-3}j_{2k-2}}g^{j_{2k-1}l}g^{j_{2k}s}.
\end{equation}
The tensor $\tilde P_{(k)}$ has also the  crucial property of divergence-free which enables us to define a new mass.

Let us assume now that $2\le k <\frac n2$. Remark that the case $k=1$ was studied  in \cite{Wang, CH, ZhangX, AnderssonCaiGalloway, CN, DGS, dLG1, dLG, HuangWu1,
BHW, LWX} mentioned above and we will not repeat this remark
 again later.

Let us first give the definition of  asymptotically hyperbolic manifolds.
Throughout this paper, we denote the hyperbolic metric  by $(\mathbb{H}^n,b=dr^2+\sinh^2r d\Theta^2),$
where $d\Theta^2$ is the standard round metric  on $\S^{n-1}$.

A Riemannian manifold $(\cM^n,g)$ is called {\it asymptotically hyperbolic of decay order $\tau$} if there exists a compact subset $K$ and a diffeomorphism at infinity $\Phi:\cM\setminus K\rightarrow\mathbb{H}^n\setminus B$, where $B$ is a closed ball in $\mathbb{H}^n$, such that $(\Phi^{-1})^{\ast}g$ and $b$ are uniformally equivalent on $\H^n\setminus B$ and
\begin{equation}\label{decaytau}
\|(\Phi^{-1})^{\ast}g-b\|_b+\|\bar\nabla\left((\Phi^{-1})^{\ast}g\right)\|_b+ \|\bar\nabla^2\left((\Phi^{-1})^{\ast}g\right)\|_b=O(e^{-\tau r} ),
\end{equation}
where $\bar\nabla$ denotes the covariant derivative with respect to the hyperbolic metric $b$.

In the first part of this paper
we first introduce  a  ``higher order" mass for asymptotically hyperbolic manifolds with slower decay.

\begin{defi}\label{mk_AH}
Assume that  $(\cM^n,g)$ is  an asymptotically hyperbolic manifold of decay order $$\tau>\frac{n}{k+1},$$ and $V \tilde L_k$ is integrable on $(\cM^n,g)$ for $V\in\mathbb N_b:= \{V\in C^{\infty}(\H^n)|\Hess^b V=Vb\}$. We define the Gauss-Bonnet-Chern mass integral with respect to the diffeomorphism $\Phi$ by
\begin{equation}\label{HV_intro}
H_k^\Phi(V)=\lim_{r\rightarrow\infty}\int_{S_r}\bigg(\big(V\bar\nabla_s e_{jl}-e_{jl} \bar\nabla_s V\big) \tilde P^{ijls}_{(k)}\bigg)\nu_i d\mu,
\end{equation}
where $e_{ij}:=((\Phi^{-1})^{\ast}g)_{ij}-b_{ij}$.
\end{defi}

Remark that the hyperbolic mass $m_1^{\H}$ defined by  Wang \cite{Wang} and Cru\'sciel-Herzlich \cite{CH}
can in principle be defined for asymptotically hyperbolic manifolds of decay order
$$\tau >\frac n2.$$
For a discussion  about the range of $\tau$, see Remark \ref{rem3.4} below.

The above definition of the asymptotically hyperbolic mass involves the choice of coordinates at infinity. Hence one needs
 to ask if it is a geometric invariant, namely  if it  does not depend on the choice of  coordinates at infinity. This is true.

\begin{theo}
Suppose that $(\cM^n,g)$ is an asymptotically hyperbolic manifold of decay order $\tau>\frac{n}{k+1}$ and for $V\in\mathbb N_b$,
$V \tilde L_k$ is integrable on $(\cM^n,g)$,  then the mass functional $H_k^\Phi(V)$  is well-defined and does not depend on the choice of the  coordinates at infinity used in the definition.
\end{theo}

The proof is motivated by \cite{CH} and \cite{M} and certainly also by \cite{Bar}.
 From the mass functional $H_k^\Phi$ on $\mathbb N_b$  we define
a higher order mass, the Gauss-Bonnet-Chern mass for asymptotically hyperbolic manifolds as follows
\begin{equation}\label{GBC_h}
 m^{\H}_k:=c(n,k)\inf_{\mathbb{N}_{b}\cap\{V>0,\eta(V,V)=1\}}H_k^{\Phi}(V),  \end{equation}
 where $c(n,k)$ is the  normalization constant   given in \eqref{eq_const} and $\eta$ is a Lorentz inner product defined in (\ref{eta}) below.
A more precise definition  will be given in Definition 3.3.
As discussed in (\ref{GBC_h2}) below,
one may assume that the infimum in \eqref{GBC_h} is achieved by
 $$V=V_{(0)}=\cosh r,$$
where $r$ is the hyperbolic distance to a fixed point $x_0\in \H^n$. Therefore,  in the rest of the introduction, we fix
 $V=V_{(0)}=\cosh r.$

\begin{theo}[Positive Mass Theorem]\label{PMthm}
Let $(\cM^n,g)=(\mathbb{H}^n,b+V^2df\otimes df)$ be the graph of a smooth  function $f:\mathbb{H}^n\rightarrow \mathbb{R}$  which satisfies $V\tilde L_k$ is integrable and  $(\cM^n,g)$ is asymptotically hyperbolic of decay order $\tau>\frac{n}{k+1}$.
 Then we have
\beq\label{mass_form}
m_k^{\H}=c(n,k)\int_{\cM^n}\frac 12\frac{V\tilde L_k}{\sqrt{1+V^2|\bv f|^2}}dV_g.
\eeq
In particular, $\tilde L_k\geq 0$ implies $m_k^{\H}\geq 0.$
\end{theo}
The condition
\begin{equation}\label{D_con}
\tilde L_k\geq 0,
\end{equation} is a dominant energy  condition like \eqref{D_con1}. In fact, one can check when $k=1$,
$\tilde L_1=R+n(n-1)$. When $k=2$, $\tilde L_2=L_2+2(n-2)(n-3)R+n(n-1)(n-2)(n-3).$  By the definition of  $\widetilde L_k$ (see \eqref{tildeLk_intro}
and \eqref{Riem})
it is trivial  to see that  $\widetilde L_k(b)$ of the standard hyperbolic metric $b$ vanishes for all $k$. Hence \eqref{D_con} is the same as
$\widetilde L_k\ge \widetilde L_k(b)$.

Such a  beautiful expression  \eqref{mass_form} was found first by Lam for the ADM mass for asymptotically flat graphs over $\R^n$, and was generalized
for the Gauss-Bonnet-Chern mass in \cite{GWW}.
Dahl-Gicquaud-Sakovich \cite{DGS} obtained this formula for $m^{\H}_1$  for asymptotically hyperbolic graphs in $\H^n$. See also the work of de Lima-Gir\~ao \cite{dLG} and
Huang-Wu \cite{HuangWu}.

 Furthermore, if the manifold is an asymptotically hyperbolic graph with a horizon boundary, we establish a relationship
 between  our new  mass and  a weighted higher order mean curvature
 in the following.
\begin{theo}\label{re}
Let $\Omega$ be a bounded open set in $\mathbb{H}^n$ with boundary $\Sigma=\partial\Omega$, and  let $(\cM^n,g)=(\mathbb{H}^n\backslash \Omega,b+V^2df\otimes df)$  an asymptotically hyperbolic graph of decay order $\tau>\frac{n}{k+1}$ with a horizon boundary $\Sigma$
satisfying that $V\tilde L_k$ is integrable over $\cM^n$.
Assume that each connected component of $\Sigma$ is in a level set of $f$ and $|\bar\nabla f(x)|\rightarrow\infty$ as
$x\rightarrow\Sigma$.
Then
$$m_k^{\H}=c(n,k)\bigg(\frac{1}{2}\int_{\cM^n}\frac{V\tilde L_k}{\sqrt{1+V^2|\bar\nabla f|^2}}dV_g+ \frac{(2k-1)!}{2}\int_{\Sigma}V\sigma_{2k-1} d\mu\bigg),$$
where $\sigma_k$ denotes $k$-th mean curvature of $\Sigma$ induced by the hyperbolic metric $b$.
 \end{theo}

For the precise definition of asymptotically hyperbolic graphs with a horizon boundary, see Section 5 below.

In order to obtain a Penrose type inequality for the hyperbolic mass $m_k^{\mathbb H}$ for asymptotically hyperbolic graphs with a horizon, we need
to establish  a ``weighted" hyperbolic Alexandrov-Fenchel inequality. This is the objective of the second part of this paper.
To state this inequality, let us first introduce a definition.
A hypersurface in $\H^n$ is {\it horospherical convex}
 if all principal curvatures are larger than or equal to $1$.
 The horospherical convexity is a natural geometric concept, which is equivalent to
 the geometric convexity in the  hyperbolic space.

\begin{theo}\label{Alexandrov}
Let  $\Sigma$ be a horospherical convex hypersurface in the hyperbolic space $\H^n$. We have
\begin{equation}\label{eq1_thm1.5}
  \int_{\Sigma} V\s_{2k-1}  d\mu \ge{C_{n-1}^{2k-1}} \omega_{n-1}{\left(\left(\frac{|\Sigma|}{\omega_{n-1}}\right)^{\frac{n}{k(n-1)}}+\left(\frac{|\Sigma|}{\omega_{n-1}}\right)^{\frac{n-2k}{k(n-1)}} \right)}^{k}.
\end{equation}
Equality holds if and only if $\Sigma$ is a geodesic sphere centered at $x_0$ in $\H^n$.
\end{theo}

When $k=1$, inequality \eqref{eq1_thm1.5} is just \eqref{eq_dLG}, which was proved by de Lima and Gir\~ao in \cite{dLG}.
These inequalities have their own interest in  integral geometry as well as in  differential geometry. Recently,
another type of Alexandrov-Fenchel inequalities in $\H^n$ without the weighted $V$ has been established in \cite{LWX, GWW_AF2, GWW_AFk, WX}.
 We will give a short introduction about Alexandrov-Fenchel inequalities in Section 6 below.

As a consequence of Theorem \ref{re} and Theorem \ref{Alexandrov},
the Penrose inequality for  the Gauss-Bonnet-Chern mass $m_k^{\H}$ for asymptotically hyperbolic graphs  with horizon boundaries follows.

\begin{theo}[Penrose Inequality]\label{thm_P} Assume that all conditions given in Theorem 1.4 hold.
If each  connected component of $\Sigma$ is horospherical convex,
 then
\beq\label{Pen_k} m_k^{\H}\geq
\frac{1}{2^k}{\left(\left(\frac{|\Sigma|}{\omega_{n-1}}\right)^{\frac{n}{k(n-1)}}+\left(\frac{|\Sigma|}{\omega_{n-1}}\right)^{\frac{n-2k}{k(n-1)}} \right)}^{k},\eeq
provided that
$$\tilde L_k\geq0.$$
Moreover, equality is achieved by the anti-de Sitter Schwarzschild type metric
\begin{equation}\label{metric}
g_{\rm adS-Sch}=(1+\rho^2-\frac{2m}{\rho^{\frac nk-2}})^{-1}d\rho^2+\rho^2d\Theta^2,\quad \rho=\sinh r.
\end{equation}
\end{theo}
Remark that metric \eqref{metric} can be represented as a graph in $\H^{n+1}$. See Section 5 below.
Motivated by the previous results, it is natural to propose the following conjecture.
\begin{conj} Let $k<\frac n2$ and
let  $\cM^n$ be a complete asymptotically hyperbolic manifold of decay order $\tau>\frac{n}{k+1}$ and  with finite integral $\int_{\cM}V|\tilde L_k|d\mu$.
Assume that the dominant energy condition
\[\tilde L_k\ge 0,\]
holds, then
the new mass for asymptotically hyperbolic manifolds
\be \label{mass_con}m^\H_k\ge 0.\ee
Moreover, if $\cM$ has a horizon type boundary $\Sigma$, then the Penrose inequality
\beq\label{Pen_k_con} m_k^{\H}\geq
\frac{1}{2^k}{\left(\left(\frac{|\Sigma|}{\omega_{n-1}}\right)^{\frac{n}{k(n-1)}}+\left(\frac{|\Sigma|}{\omega_{n-1}}\right)^{\frac{n-2k}{k(n-1)}} \right)}^{k},\eeq
holds. Furthermore, the rigidity theorem holds. More precisely, equality  in \eqref{mass_con} implies that $\cM$ is isometric to the standard hyperbolic space $\H^n$
and equality  in \eqref{Pen_k_con} implies that $\cM$ is  isometric to the anti de-Sitter Schwarzschild type metric \eqref{metric}
 outside its corresponding horizon.
\end{conj}

\

The rest of this paper is organized as follows. In Section 2, some preliminaries are reviewed. This section is divided into two subsections. In Subsection 2.1, we recall the mass of asymptotically hyperbolic manifolds defined in \cite{CH,CN}.  In Subsection 2.2, some basic facts of the Gauss-Bonnet curvature
as well as some basic properties of the $k$-th mean curvature are outlined. In Section 3, we introduce the
 higher order mass for asymptotically hyperbolic manifolds with slower decay by using the Gauss-Bonnet curvature. Its geometric invariance is also proved in Section 3.
For this new defined mass, we show a positive mass theorem for asymptotically hyperbolic graphs in Section 4 and establish a relationship between the corresponding Penrose type inequality and ``weighted'' Alexandrov-Fenchel inequalities in the hyperbolic space in Section 5. Moreover, an interesting example,
a Schwarzschild type metric,
 is also given in this section. In Section 6, we give a short introduction for the  Alexandrov-Fenchel inequalities with or without weight
 in the hyperbolic space. As a preparation to establish these  Alexandrov-Fenchel inequalities, several important Minkowski integral formulas
 between integrals involving $\sigma_k$ are proved in Section 7.
 Section 8 is devoted to discuss our crucial Minkowski integral formulas between integrals involving $\sigma_k$ with the weight (see (\ref{ineq1}) below) by using a ``conformal'' flow.  We establish ``weighted'' Alexandrov-Fenchel inequalities in Section 9,  which
  implies  optimal Penrose type inequalities for the new mass for asymptotically hyperbolic graphs with a horizon type boundary.
We include  calculations of the modified Gauss-Bonnet curvature $\tilde L_k$ for the anti-de Sitter-Schwarzschild  type metric (\ref{metric}) as well as its Gauss-Bonnet-Chern mass $m_k^{\H}$ in Appendix A. We show formula \eqref{ast}  in Appendix B.

\section{Preliminaries}
\subsection{Asymptotically hyperbolic manifolds and the mass for $k=1$}
In this subsection, we recall the mass of asymptotically hyperbolic manifolds from Chru\'sciel and Herzlich \cite{CH} which coincides with the one defined by Wang \cite{Wang} for the special case of conformally compact manifolds. Throughout this paper, we denote the hyperbolic metric  by $(\mathbb{H}^n,b=dr^2+\sinh^2r d\Theta^2),$
where $d\Theta^2$ is the standard round metric  on $\S^{n-1}$.
We recall the definition of asymptotically hyperbolic manifolds.
\begin{defi}\label{AH}
A Riemannian manifold $(\cM^n,g)$ is called asymptotically hyperbolic of decay order $\tau$ if there exists a compact subset $K$ and a diffeomorphism at infinity $\Phi:\cM\setminus K\rightarrow\mathbb{H}^n\setminus B$, where $B$ is a closed ball in $\mathbb{H}^n$, such that $(\Phi^{-1})^{\ast}g$ and $b$ are uniformly equivalent on $\H^n\setminus B$ and
\begin{equation}\label{decaytau2}
\|(\Phi^{-1})^{\ast}g-b\|_b+\|\bar\nabla\left((\Phi^{-1})^{\ast}g\right)\|_b+\|\bar\nabla^2\left((\Phi^{-1})^{\ast}g\right)\|_b=O(e^{-\tau r} ).
\end{equation}
\end{defi}
Note that when $k=1$, one needs no decay condition on the second derivatives of $(\Phi^{-1})^{\ast}g$, see \cite{CH}.
Set
\begin{equation}\label{Nb}
\mathbb{N}_b:= \{V\in C^{\infty}(\H^n)|\Hess^b V=Vb\}.
\end{equation}
Any element $V$ in $\mathbb{N}_b$ has a nice property  that the Lorentzian metric $\gamma=-V^2dt^2+b$ is a static solution of the Einstein equation $Ric(\gamma)+n\gamma=0.$
 $\mathbb{N}_b$ is an $(n+1)$-dimensional vector space spanned by   an orthonormal basis of the following functions
$$V_{(0)}=\cosh r,\; V_{(1)}=x^1\sinh r,\;\cdots,\; V_{(n)}=x^n\sinh r,$$
where $r$ is the hyperbolic distance from an arbitrary fixed point on $\H^n$ and  $x^1,x^2,\cdots, x^n$ are the coordinate functions  restricted to $\S^{n-1}\subset \R^n$.
We   equip the vector space $\mathbb{N}_b$  with a Lorentz inner product $\eta$ with signature $(+,-,\cdots,-)$ such that
\begin{equation}\label{eta}
\eta(V_{(0)}, V_{(0)})=1,\qquad \mbox{and}\quad\eta(V_{(i)}, V_{(i)})=-1\quad\mbox{for}\quad i=1,\cdots,n.
\end{equation}
It is clear that the subset $\mathbb{N}_b^+$ of positive functions is just the interior of the future
lightcone. Let $\mathbb{N}_b^1$ be the subset of $\mathbb{N}_b^+$  of functions $V$ with $\eta(V,V)=1$. One can check easily that
every function $V$ in $\mathbb{N}_b^1$ has the following form
\[V=\cosh {\rm dist}_b(x_0, \cdot),\]
for some $x_0\in \H^n$, where ${\rm dist}_b$ is the distance function with respect  to the metric $b$.

Let $(\cM^n,g)$ be an asymptotically hyperbolic manifold of decay order $\tau>\frac n2$ and
\begin{equation}\label{condition}
\int_{
\cM}\cosh r\;|R+n(n-1)| d\mu<\infty,
\end{equation}
where $R$ is the scalar curvature with respect to the metric $g$. Then the mass functional of $(\cM^n,g )$ with respect to $\Phi$  on $\mathbb{N}_b$  is defined by
\begin{equation}\label{AHmass}
H_{\Phi}(V)=\lim_{r\rightarrow\infty}\int_{S_r} \bigg(V(\div^b e-d\tr^b e)+(\tr^b e) dV-e(\nabla^b V,\cdot)\bigg)\nu d\mu,
\end{equation}
where $e:= \Phi_{\ast}g-b$, $S_r$ is a geodesic sphere with radius $r$, $\nu$ is the outer normal of $S_r$
and $d\mu$ is the area  element with respect to  the induced metric on $S_r$.
From\cite{CH}, the  limit in (\ref{AHmass}) exists and is finite, provided  that the decay condition $\tau>\frac n2$ and the integrable condition (\ref{condition}) holds.
  If $A$ is an isometry of the hyperbolic metric $b$ and $\Phi$ a chart as defined in Definition 2.1, then one can easily check
  that   $A\circ\Phi$ is also such a chart and
\begin{equation}\label{relation0}
H_{A\circ\Phi}(V)=H_{\Phi}(V\circ A^{-1}),
\end{equation}
holds.
 Moreover, if $\Psi$ is another chart satisfying conditions in Definition 2.1, there is an isometry $A$ of $b$ such that $\Psi=A\circ \Phi$ modulo lower order terms \cite {CH, Herzlich} such that the mass functional is not changed. It follows that  the defined limit (\ref{AHmass}) is a geometric invariant independent of the choice of coordinates at infinity. See also Theorem $3.3$ in \cite{M}. If the mass functional $H_{\Phi}:\mathbb{N}_b\rightarrow\R$ is timelike future directed, i.e.,
  $H_{\Phi}(V)>0$ for all $V\in \mathbb{N}^+$,
then we can define  the asymptotically hyperbolic mass  $m^{\H}$ for asymptotically hyperbolic manifold of decay order $\tau>\frac n2$ as follows:
 \begin{equation}\label{m_H}
m^{\H}_1:=\frac{1}{2(n-1)\omega_{n-1}}\;\inf_{\mathbb{N}_{b}^1}H_{\Phi}(V).
 \end{equation}

 From above discussion we know that
 \[m_1^{\H}=\frac{1}{2(n-1)\omega_{n-1}}\;\inf_{\Phi} H_{\Phi}(V_{(0)}),
 \]
 where  the infimum takes over of asymptotically hyperbolic charts $\Phi$ as above.
 Hence to estimate $m^{\H}$ we only need to estimate $H_\Phi(V_{(0)})$ for a fixed $V_{(0)}$ and for all $\Phi$ and $ V_{(0)}$
 is a fixed function in $ \mathbb{N}_{b}^1$.
 This strategy will also be used to estimate our new mass defined in Section 3 below.

\subsection{Gauss-Bonnet curvature and the $k$-th mean curvature}
Recall that
\begin{equation}\label{Lk}
L_k:=\frac{1}{2^k}\d^{i_1i_2\cdots i_{2k-1}i_{2k}}
_{j_1j_2\cdots j_{2k-1}j_{2k}}{R_{i_1i_2}}^{j_1j_2}\cdots
{R_{i_{2k-1}i_{2k}}}^{j_{2k-1}j_{2k}},
\end{equation}
is the {\it $k$-th Gauss-Bonnet curvature}, or simply the {\it Gauss-Bonnet curvature}. Here
 the generalized Kronecker delta is defined by
\begin{equation}\label{generaldelta}
 \d^{j_1j_2 \cdots j_r}_{i_1i_2 \cdots i_r}=\det\left(
\begin{array}{cccc}
\d^{j_1}_{i_1} & \d^{j_2}_{i_1} &\cdots &  \d^{j_r}_{i_1}\\
\d^{j_1}_{i_2} & \d^{j_2}_{i_2} &\cdots &  \d^{j_r}_{i_2}\\
\vdots & \vdots & \vdots & \vdots \\
\d^{j_1}_{i_r} & \d^{j_2}_{i_r} &\cdots &  \d^{j_r}_{i_r}
\end{array}
\right).
\end{equation}
 This curvature was first appeared in the paper of Lanczos \cite{Lan} in 1938 for $n=4$ and $k=2$ 
 and
 has been intensively studied in the Gauss-Bonnet gravity,
which is a generalization of the Einstein gravity.  $L_1$ is just the scalar curvature. One could check  that $L_k=0$ if $2k>n$.
When $2k=n$, $L_k$ is in fact the Euler density, which was studied by Chern \cite{Chern1, Chern2} in his proof of the Gauss-Bonnet-Chern theorem. See also a nice survey \cite{Zhang}
on the  proof of the Gauss-Bonnet-Chern theorem.
 For
$k<n/2$, $L_k$ is therefore called the dimensional continued Euler density
in physics.
  One can decompose the Gauss-Bonnet curvature in the following way
\begin{equation}
L_k=P_{(k)}^{stjl}R_{stjl},
\end{equation}
where \begin{equation}\label{Pk}
P_{(k)}^{st lj}:=\frac{1}{2^k}\d^{i_1i_2\cdots i_{2k-3}i_{2k-2}st}
_{j_1j_2\cdots  j_{2k-3}j_{2k-2} j_{2k-1}j_{2k}}{R_{i_1i_2}}^{j_1j_2}\cdots
{R_{i_{2k-3}i_{2k-2}}}^{j_{2k-3}j_{2k-2}}g^{j_{2k-1}l}g^{j_{2k}j}.
\end{equation}
(Throughout this paper, we adopt the Einstein summation convention.)
For example, for the case $k=1,$
\begin{equation}\label{L1}
L_1=R=R_{stjl}P^{stjl}_{(1)},
\end{equation}
where
\begin{equation}\label{P1}
P_{(1)}^{stjl}=\frac 12(g^{sj}g^{tl}-g^{sl}g^{tj}).
\end{equation}
For the case $k=2$, we have
\begin{equation}\label{L2}
L_2=\frac 14 \delta
^{i_1i_2i_3i_4}_{j_1j_2j_3j_4}{R^{j_1j_2}}_{i_1i_2}
{R^{j_3j_4}}_{i_3i_4}=\| Rm\|^2-4\| Ric\|^2+ R^2
\end{equation}
and
\begin{equation}\label{P2}
P_{(2)}^{stjl}= R^{stjl}+ R^{tj}g^{sl}- R^{tl}g^{sj}- R^{sj}g^{tl}+ R^{sl}g^{tj}+\frac12  R(g^{sj}g^{tl}-g^{sl}g^{tj}).
\end{equation}
We collect important properties  of the four-tensor $P_{(k)}$ in the following lemma.
\begin{lemm}\label{PProperty}
\begin{enumerate}[(1)]
\item $P_{(k)}$ shares the same symmetry and antisymmetry with the Riemann curvature tensor that
\begin{equation}\label{symmetry}
P_{(k)}^{stjl}=-P_{(k)}^{tsjl}=-P_{(k)}^{stlj}=P_{(k)}^{jlst}.
\end{equation}
\item $P_{(k)}$  satisfies the first Bianchi identity, i.e., $P^{stjl}+P^{tjsl}+P^{jstl}=0$.
\item  $P_{(k)}$ is divergence-free,
\begin{equation}\label{div-free}
\nabla_s P_{(k)}^{stjl}=0.
\end{equation}
\end{enumerate}
\end{lemm}
\begin{proof} (1) is easy to check.  For (2), one can also  calculate it directly, or see page 26 in \cite{Willa}. Here we provide a proof for (3), since
this property is crucial.
 To show (\ref{div-free}), we use the differential Bianchi identity
$$\nabla_s {R_{i_1 i_2}}^{j_1 j_2}=-\nabla_{i_1}{R_{i_2 s}}^{j_1 j_2}-\nabla_{i_2}{R_{s i_1}}^{j_1 j_2},$$
and the fact of the symmetry property of the generalized Kronecker delta and the curvature tensor. From (\ref{Pk}), we compute
\begin{eqnarray*}
&&\nabla_s P_{(k)}^{st jl}\\
&=&\frac{k-1}{2^k}\d^{i_1i_2\cdots i_{2k-3}i_{2k-2}st}
_{j_1j_2\cdots  j_{2k-3}j_{2k-2} j_{2k-1}j_{2k}}(\nabla_s {R_{i_1i_2}}^{j_1j_2})\cdots
{R_{i_{2k-3}i_{2k-2}}}^{j_{2k-3}j_{2k-2}}g^{j_{2k-1}j}g^{j_{2k}l}\\
&=&\frac{k-1}{2^k}\d^{i_1i_2\cdots i_{2k-3}i_{2k-2}st}
_{j_1j_2\cdots  j_{2k-3}j_{2k-2} j_{2k-1}j_{2k}}\big(-\nabla_{i_1}{R_{i_2 s}}^{j_1 j_2}-\nabla_{i_2}{R_{s i_1}}^{j_1 j_2}\big)\cdots
{R_{i_{2k-3}i_{2k-2}}}^{j_{2k-3}j_{2k-2}}g^{j_{2k-1}j}g^{j_{2k}l}\\
&=&\frac{k-1}{2^k}\d^{i_1i_2\cdots i_{2k-3}i_{2k-2}st}
_{j_1j_2\cdots  j_{2k-3}j_{2k-2} j_{2k-1}j_{2k}}\big(-2\nabla_{i_1}{R_{i_2 s}}^{j_1 j_2}\big)\cdots
{R_{i_{2k-3}i_{2k-2}}}^{j_{2k-3}j_{2k-2}}g^{j_{2k-1}j}g^{j_{2k}l}\\
&=&\frac{k-1}{2^k}\d^{i_1i_2\cdots i_{2k-3}i_{2k-2}st}
_{j_1j_2\cdots  j_{2k-3}j_{2k-2} j_{2k-1}j_{2k}}\big(-2\nabla_{s}{R_{i_1 i_2}}^{j_1 j_2}\big)\cdots
{R_{i_{2k-3}i_{2k-2}}}^{j_{2k-3}j_{2k-2}}g^{j_{2k-1}j}g^{j_{2k}l}\\
&=&-2\nabla_s P_{(k)}^{st jl},
\end{eqnarray*}
which implies
$$\nabla_s P_{(k)}^{st jl}=0.$$ Thus we complete the proof.
\end{proof}
This Lemma, in particular, the divergence-free property \eqref{div-free}, plays a crucial role in this paper.

\vspace{3mm}
In the following, we recall several basic properties of the $k$-th mean curvature.
Let $\s_k$ be the $k$-th elementary symmetry function $\s_k:\R^{n-1}\to \R$ defined by
\[\s_k(\Lambda)=\sum_{i_1<\cdots<i_{k}}\lambda_{i_1}\cdots\lambda_{i_k}\quad  \hbox{ for } \Lambda=(\l_1, \cdots,\l_{n-1})\in \R^{n-1}.\]
This definition can be easily extended to symmetric matrices.
 For a symmetric matrix $B$, denote $\lambda(B)=(\lambda_1(B),\cdots,\lambda_{n-1}(B))$ be the eigenvalues of $B$. We set
\[
\s_k(B):=\s_k(\lambda(B)).
\]
Here we make a convention that $\s_0(B)=1$. Let $I$ be the identity matrix. Then we have for any $t\in\R$,
\[
\s_{n-1}(I+tB)=\det(I+tB)=\sum_{i=0}^{n-1}\s_i(B)t^i.
\]
 The Garding cone $\Gamma_k^+$ is defined as
\[\Gamma_k^+=\{\Lambda \in \R^{n-1} \, |\,\s_j(\Lambda)>0, \quad\forall j\le k\}.\]
A symmetric matrix $B$ is called belong to $\Gamma_k^+$  if $\lambda(B)\in \Gamma_k^+$.
We collect the basic facts about $\s_k$, which will be directly used in this paper. For further details, we refer to a survey of Guan \cite{Guan}.

The $k$-th Newton transformation is defined as follows
\begin{equation}\label{Newtondef}
(T_k)^{i}_{j}(B):=\frac{\partial \s_{k+1}}{\partial B^{j}_{i}}(B),
\end{equation}
where $B=(B^{i}_{j})$. If there is no confusion, we may omit the index $k$. We recall basic formulas about $\s_k$ and $T$.\\
\begin{eqnarray}\label{sigmak}
\s_k(B)& =&\ds\frac{1}{k!}\d^{i_1\cdots i_k}
_{j_1\cdots j_k}B_{i_1}^{j_1}\cdots
{B_{i_k}^ {j_k}}=\frac{1}{k} \tr(T_{k-1}(B)B),
\\
(T_k)^{i}_j(B) & =& \ds \frac{1}{k!}\d^{ii_1\cdots i_{k}}
_{j j_1\cdots j_{k}}B_{i_1}^{j_1}\cdots
{B_{i_{k}}^{j_{k}}}\label{Tk}.
\end{eqnarray}

Let $(\Sigma,\gamma)$ be a hypersurface in $\H^n$ with $\gamma$ the induces metric and $B$ its second fundamental form. The $k$-th mean curvature of $\Sigma$ is defined
\begin{equation}\label{add_sigmak}
\s_k=\s_k(B).
\end{equation}
When $k=1$, $\s_1$ is just the ordinary mean curvature.

\section{A new mass for asymptotically hyperbolic manifolds}
In this section, we will introduce a higher order mass for asymptotically hyperbolic manifolds which extends the mass defined in Subsection 2.1.
Inspired by our recent work about a new mass on asymptotically flat  manifolds \cite{GWW}, which was  in turn motivated by the Gauss-Bonnet gravity \cite{DT1, DT2}, we define it by using the Gauss-Bonnet curvature $L_k$. As indicated in  Introduction, we need to modify the
 Gauss-Bonnet curvature $L_k$ a little bit in order to obtain a new one.
More precisely, on a Riemannian manifold $(\cM^n,g)$, we define a new four-tensor
\begin{equation}\label{tildeRm}
\widetilde  {\rm Riem}_{ijsl}(g)=\tilde  R_{ijsl}(g):=R_{ijsl}(g)+g_{is}g_{jl}-g_{il}g_{js}.
\end{equation}
Let us apply the same argument in Subsection 2.2 to define a new Gauss-Bonnet curvature with respect to this tensor $\widetilde {Riem}$
\begin{equation}\label{tildeLk}
\tilde L_k:=\frac{1}{2^k}\d^{i_1i_2\cdots i_{2k-1}i_{2k}}
_{j_1j_2\cdots j_{2k-1}j_{2k}} {\tilde R_{i_1i_2}}^{\quad\, j_1 j_2}\cdots
{\tilde R_{i_{2k-1}i_{2k}}}^{\qquad\;\,j_{2k-1}j_{2k}}=\tilde R_{stjl}{\tilde P_{(k)}}^{stjl},
\end{equation}
where
\begin{equation}\label{tildePk}
{{\tilde P}_{(k)}}^{st jl}:=\frac{1}{2^k}\d^{i_1i_2\cdots i_{2k-3}i_{2k-2}st}
_{j_1j_2\cdots  j_{2k-3}j_{2k-2} j_{2k-1}j_{2k}}{\tilde R_{i_1i_2}}^{\quad\, j_1 j_2}\cdots
{\tilde R_{i_{2k-3}i_{2k-2}}}^{\qquad\quad j_{2k-3}j_{2k-2}}g^{j_{2k-1}j}g^{j_{2k}l}.
\end{equation}
It is easy to see that
\[\tilde L_1=L_1+n(n-1)=R+n(n-1)\]
and
\[ \tilde L_2= L_2+2(n-2)(n-3)R+n(n-1)(n-2)(n-3).\]
Moreover, we have
\begin{equation}\label{tildeP1}
\tilde P_{(1)}^{stjl}= P_{(1)}^{stjl}=\frac{1}{2}(g^{sj}g^{tl}-g^{sl}g^{tj}),
\end{equation}
and
\begin{equation}\label{tildeP2}
\tilde P_{(2)}^{stjl}=\tilde R^{stjl}+\tilde R^{tj}g^{sl}-\tilde R^{sj}g^{tl}-\tilde R^{tl}g^{sj}+\tilde R^{sl}g^{tj}+\frac12 \tilde R(g^{sj}g^{tl}-g^{sl}g^{tj}).
\end{equation}
where
\[\tilde R^{sj} =g_{tl}  \tilde R^{stjl}, \quad \tilde R=g_{sj} \tilde R^{sj}. \]
By a direct calculation, one gets that
\begin{equation}\label{relation}
\tilde P_{(2)}^{stjl}=P_{(2)}^{stjl}+(n-2)(n-3)P_{(1)}^{stjl},
\end{equation}
where $P_2$ and $P_1$ are defined in (\ref{P2}) and (\ref{P1}) respectively. Remark that  $\tilde P_{(k)}$ is
a ``good" linear combination of $P_{(i)}, \; i=1,\cdots,k$ with constant coefficients only depending on $n$ and $k$.
It is clear that the tensor $\tilde P_{(k)}$ satisfies the same  properties as  $P_{(k)}$, 
namely,
\begin{equation}\label{antisymmetry}
\tilde P^{stjl}_{(k)}=-\tilde P^{tsjl}_{(k)}=-\tilde P^{stlj}_{(k)}=\tilde P^{jlst}_{(k)},
\end{equation}
\begin{equation}\label{Bianchi}
\tilde P^{stjl}_{(k)}+\tilde P^{tjsl}_{(k)}+\tilde P^{jstl}_{(k)}=0
\end{equation}
and
\begin{equation}\label{divergence-free}
\nabla_s\tilde P^{stjl}_{(k)}=0.
\end{equation}
In the following, if there is no confusion, we just write $\tilde P=\tilde P_{(k)}$ for brief
and denote the tensors associated with the hyperbolic metric $b$ with a bar.
In view of (\ref{antisymmetry}) and (\ref{divergence-free}), we observe that for asymptotically hyperbolic manifolds of decay order $\tau$,
 the Gauss-Bonnet curvature multiplying with  a function $V\in\{\mathbb{N}_b\}$, i.e., $V\tilde L_k$,
   can be expressed as a divergence term together with a term of  faster decay.

First, since the difference of two Christoffel symbols is a tensor, we have the following formula
\begin{eqnarray}\label{Gammadiff}
\Gamma_{ij}^s-\bar{\Gamma}_{ij}^s&=&\frac 12 g^{sl}(\bar{\nabla}_i g_{lj}+\bar{\nabla}_j g_{li}-\bar{\nabla}_l g_{ij})\nonumber\\
&=&\frac 12 g^{sl}(\bar{\nabla}_i e_{lj}+\bar{\nabla}_j e_{li}-\bar{\nabla}_l e_{ij}),
\end{eqnarray}
where $e_{ij}:=g_{ij}-b_{ij}$ in local coordinates.
And the difference of two curvature tensors in local coordinates is
\begin{equation}\label{Rmdiff}
\begin{array}{rcl} \ds\vs  R_{ij\;\; l}^{\;\;\,\, t}-\bar R_{ij\;\;l}^{\;\;\,\,t}
&=& \ds {\nabla}_i(\Gamma_{jl}^t-\bar{\Gamma}_{jl}^t)-{\nabla}_j(\Gamma_{il}^t-\bar{\Gamma}_{il}^t) \\
&&\ds
 +(\Gamma_{is}^m-\bar{\Gamma}_{is}^m)
(\Gamma_{jl}^s-\bar{\Gamma}_{jl}^s)
-(\Gamma_{js}^m-\bar{\Gamma}_{js}^m)(\Gamma_{il}^s-\bar{\Gamma}_{il}^s).\end{array}
\end{equation}
We emphasize again that $\bar \nabla$ is the covariant derivative w.r.t to $b$ and $\nabla$ the covariant derivative w.r.t to $g$.
By the divergence-free property of $\tilde P$ and the fact that the
quadratic terms of Christoffel symbols have a faster decay, we have from (\ref{antisymmetry}), (\ref{divergence-free}) and  (\ref{Gammadiff}) that
\begin{eqnarray}\label{decayestimate}
V\tilde L_k&=&V\tilde P^{ijsl}\tilde R_{ijsl}=V\tilde P^{ijsl}\bigg(R_{ijsl}+g_{is}g_{jl}-g_{il}g_{js}\bigg)\nonumber\\
&=&V\tilde P^{ijsl}\bigg(g_{st}( R_{ij\;\; l}^{\;\;\,\, t}-\bar R_{ij\;\;l}^{\;\;\,\,t})+ g_{is}g_{jl}-g_{il}g_{js}+ g_{st}\bar R_{ij\;\;\,l}^{\;\;\,\,t}\bigg)\nonumber\\
&=&V\tilde P^{ijsl}\bigg(g_{st}( R_{ij\;\; l}^{\;\;\,\, t}-\bar R_{ij\;\;l}^{\;\;\,\,t})+ g_{is}(g_{jl}-b_{jl})-g_{js}(g_{il}-b_{il})\bigg)\nonumber\\
&=&2V\tilde P^{ijsl}\bigg(g_{st}{\nabla}_i(\Gamma_{jl}^t-\bar{\Gamma}_{jl}^t)+g_{is}e_{jl}\bigg)+O(e^{(-(k+1)\tau+1)r})\nonumber\\
&=&V\nabla_i\bigg((\bar{\nabla}_l e_{js}+\bar{\nabla}_j e_{sl}-\bar{\nabla}_s e_{jl})\tilde P^{ijsl}\bigg)+2V\tilde P^{ijsl}b_{is}e_{jl}+O(e^{(-(k+1)\tau+1)r})\nonumber\\
&=&2V\bar{\nabla}_i(\bar{\nabla}_l e_{js}\tilde P^{ijsl})-2V\tilde P^{ijsl}b_{il}e_{js}+O(e^{(-(k+1)\tau+1)r})\nonumber\\
&=&2\bar{\nabla}_i(V(\bar{\nabla}_l e_{js})\tilde P^{ijsl})-2\bar{\nabla}_i V (\bar{\nabla}_l e_{js})\tilde P^{ijsl}-2V\tilde P^{ijsl}b_{il}e_{js}+O(e^{(-(k+1)\tau+1)r})\nonumber\\
&=&2\bar{\nabla}_i(V(\bar{\nabla}_l e_{js})\tilde P^{ijsl})-2\bar{\nabla}_l V (\bar{\nabla}_i e_{js})\tilde P^{ijsl}-2V\tilde P^{ijsl}b_{il}e_{js}+O(e^{(-(k+1)\tau+1)r})\nonumber\\
&=&2\bar{\nabla}_i\bigg((V\bar{\nabla}_l e_{js}-e_{js}\bar{\nabla}_l V)\tilde P^{ijsl}\bigg)+2(\bar{\nabla}_i\bar{\nabla}_l V-V b_{il})e_{js}\tilde P^{ijsl}+O(e^{(-(k+1)\tau+1)r}).
\end{eqnarray}
Here we have used the fact for any tensor field $T$
\begin{equation}\label{fact_decay}
\bar\nabla T-\nabla T=O(e^{-\tau r}|T|),
\end{equation}
which follows from (\ref{Gammadiff}).
Since $V\in\mathbb{N}_b$, which is defined in \eqref{Nb},  the second term in (\ref{decayestimate}) vanishes, and hence we conclude

\beq\label{add1}
V\tilde L_k=2\bar\nabla_i\bigg((V\bar{\nabla}_l e_{js}-\bar\nabla_l V e_{js})\tilde P^{ijsl}\bigg)+O(e^{(-(k+1)\tau+1)r}).\eeq
With this crucial expression, one can check that the limit
\begin{equation}\label{HV}
H_{k}^\Phi(V)=\lim_{r\rightarrow\infty}\int_{S_r}\bigg(\big(V\bar\nabla_l e_{js}-e_{js} \bar\nabla_l V\big) \tilde P^{ijsl}_{(k)}\bigg)\nu_id\mu,
\end{equation}
exists and is finite provided that $V\tilde L_k$ is integrable and the decay order $\tau>\frac{n}{k+1}$. Here $\Phi$ is any diffeomorphism used in Definition \ref{AH}. We remark that by (\ref{tildeP1}), one can check that $2H_1^\Phi(V)$ coincides with (\ref{AHmass}) defined by Chru\'{s}ciel and Herzlich \cite{CH}.
  Now we have
\begin{theo}\label{thm1}
Suppose $(\cM^n,g)(2k <n )$ is an asymptotically hyperbolic manifold of decay order $\tau>\frac{n}{k+1}$ and for $V\in\mathbb N_b$ defined in (\ref{Nb}),
$V \tilde L_k$ is integrable in $(\cM^n,g)$, then the mass functional $H^\Phi_k(V)$ defined  in (\ref{HV}) is well-defined.
\end{theo}

\begin{proof} This follows from \eqref{add1} easily.
\end{proof}

The above definition of an asymptotically hyperbolic mass functional involves the choice of  coordinates at infinity. We need to show
that it is a geometric invariant which does not depend on the choice of coordinates  at infinity.

\begin{theo}
Assume $(\cM^n,g)$ is asymptotically hyperbolic manifold of decay order
\begin{equation}\label{decaytau_order}
\tau>\frac{n}{k+1},
\end{equation}
 and satisfies the integrable condition
\be\label{add2}
\int_\cM V|\tilde L_k|<\infty,
\ee
 then $ H_k^{\Phi}(V)$ does not depend on the choice of coordinates at infinity in the sense that if there are two maps $\Phi_1,\Phi_2$ satisfying (\ref{decaytau_order}), (\ref{add2}), then there exists an isometry $A$ of $b$, such that
 $$H^{\Phi_2}_k(V)=H^{\Phi_1}_k(V\circ A^{-1}).$$
\end{theo}
\begin{proof}
The argument follows closely  from the one given by Chru\'{s}ciel and Herzlich  in the proof of hyperbolic mass \cite{CH}. See also \cite{CN, Herzlich, M}. The key point is to realize that when changing the  coordinates at infinity,  extra terms which do not decay fast enough to have vanishing integral can be collected in a divergence of some alternative 2-vector field.

First assume $\Phi_1$ and $\Phi_2$ are two  coordinates at infinity satisfying (\ref{decaytau}) with $\tau>\frac{n}{k+1}$, then there exists an isometry $A$ of the background metric $b$ such that
$$\Phi_2-\Phi_1\circ A=o(e^{-\frac{n}{k+1}r}).$$
Hence it suffices to prove
\begin{equation}\label{aim}
H^{\Phi_1\circ A}_k(V)=H^{\Phi_2}_k(V).
\end{equation}
As in \cite{CH} (see (2.17) there), we know that there is a well-defined vector field \begin{equation}\label{decayzeta}
\zeta=o(e^{-\frac{n}{k+1}r}),
\end{equation}
 such that
\begin{equation}\label{linearize}
({\Phi_2}^{-1})^{\ast}g=({\Phi_1}^{-1})^{\ast}g+\mathscr{L}_{\zeta} ({\Phi_1}^{-1})^{\ast}g+o(e^{-\frac{2n}{k+1}r})=({\Phi_1}^{-1})^{\ast}g+\mathscr{L}_{\zeta}b +o(e^{-\frac{2n}{k+1}r}),
\end{equation}
where $\mathscr{L}$ is the Lie derivative.

For $p=1,2$, set
$$\mathbb{U}_p^i=(V\bar \nabla_l (e_p)_{js}-(e_p)_{js}\bar \nabla_l V)\tilde P^{ijsl}(g_p),$$
where $g_p=({\Phi_p}^{-1})^{\ast}g,\;\; \mbox{and}\;\;e_p=g_p-b.$
It follows from (\ref{linearize}) that
$$\mathbb{U}_2^i-\mathbb{U}_1^i=\delta\mathbb{U}^i+o(e^{-(n-1)r}),$$
where $\delta\mathbb{U}^i$ is a first-order term  by linearizing $\mathbb{U}$  at $g_1$ and the remaining terms decay sufficiently fast so that they do not contribute when integrated at infinity. It remains to compute $\delta\mathbb{U}^i$.
It is crucial to realize from (\ref{linearize}) that
\begin{equation}\label{ast}
\tilde P^{ijsl}(g_2)-\tilde P^{ijsl}(g_1)=o(e^{-\frac{k n}{k+1}r}).
\end{equation}
When $k=1$ it is easy to prove, since $\widetilde P_1^{ijsl}(g)=\frac 12 (g^{is}g^{jl}-g^{il}g^{js})$. For the general $k$,
the proof of (\ref{ast}) is not obvious.
 We provide  a detailed computation  in Appendix B for the convenience of the reader. In the following computation, we write $\tilde P=\tilde P(g_1)$ for short.
Hence we have
\begin{eqnarray}\label{eq1_sec.3}
\delta\mathbb{U}^i&=&\bigg(V\bar\nabla_l(\bar\nabla _j\zeta_s+\bar\nabla_s\zeta_j)-(\bar\nabla _j\zeta_s+\bar\nabla_s\zeta_j)\bar \nabla_l V\bigg)\tilde P^{ijsl}+o(e^{-(n-1)r})\nonumber\\
&=&\bigg(V\bar\nabla_l\bar\nabla _j\zeta_s+V\bar\nabla_l\bar\nabla_s\zeta_j-\bar\nabla _j\zeta_s\bar\nabla_l V+\bar\nabla_l\zeta_j\bar \nabla_s V\bigg)\tilde P^{ijsl}+o(e^{-(n-1)r}).
\end{eqnarray}
Here in the second equality we have used the antisymmetry (\ref{antisymmetry}) to exchange the indices $s$ and $l$ in the fourth term.
 We apply the Ricci identity to deal with the second term in (\ref{eq1_sec.3}) as follows:
\begin{eqnarray}\label{eq2_sec.3}
(\bar\nabla_l\bar\nabla_s\zeta_j-\bar\nabla_s\bar\nabla_l\zeta_j)\tilde P^{ijsl}&=&({\bar R_{lsj}}^{\quad t}\zeta_t)\tilde P^{ijsl}\nonumber\\
&=&-(b_{lj}\delta_s^t-\delta_l^tb_{sj})\zeta_t\tilde P^{ijsl}\nonumber\\
&=&-b_{lj}\zeta_s\tilde P^{ijsl}+b_{sj}\zeta_l\tilde P^{ijsl}\nonumber\\
&=&-2b_{lj}\zeta_s\tilde P^{ijsl},
\end{eqnarray}
where in the last equality, (\ref{antisymmetry}) is used again. In view of  (\ref{divergence-free}), (\ref{eq1_sec.3}) and (\ref{eq2_sec.3}),
 we derive
\begin{eqnarray}\label{eq_BHW}
\delta\mathbb{U}^i\nonumber
&=&\bigg(V\bar\nabla_l\bar\nabla _j\zeta_s+V\bar\nabla_s\bar\nabla_l\zeta_j-2b_{lj}\zeta_s V-\bar\nabla _j\zeta_s\bar\nabla_l V+\bar\nabla_l\zeta_j\bar \nabla_s V\bigg)\tilde P^{ijsl}+o(e^{-(n-1)r})\nonumber\\
&=&V\bar\nabla_l\bar\nabla _j(\zeta_s\tilde P^{ijsl})+V\bar\nabla_s\bar\nabla_l(\zeta_j\tilde P^{ijsl})-2b_{lj}\zeta_sV\tilde P^{ijsl}-\bar\nabla _j(\zeta_s\tilde P^{ijsl})\bar\nabla_l V\nonumber\\
&&+\bar\nabla_l(\zeta_j\tilde P^{ijsl})\bar \nabla_s V+o(e^{-(n-1)r})\nonumber\\
&=&\bar\nabla_l\big(V\bar\nabla _j(\zeta_s\tilde P^{ijsl})\big)-2\bar\nabla _j(\zeta_s\tilde P^{ijsl})\bar\nabla_l V-2b_{lj}\zeta_sV\tilde P^{ijsl}+\bar\nabla_s\big(V\bar\nabla_l(\zeta_j \tilde P^{ijsl})\big)+o(e^{-(n-1)r})\nonumber\\
&=&\bar\nabla_l\big(V\bar\nabla _j(\zeta_s\tilde P^{ijsl})\big)-2\bar\nabla _j(\zeta_s\tilde P^{ijsl}\bar\nabla_l V)+2\zeta_s\tilde P^{ijsl}\bar\nabla _j\bar\nabla_l V-2b_{lj}\zeta_sV\tilde P^{ijsl}\nonumber\\
&&+\bar\nabla_s\big(V\bar\nabla_l(\zeta_j \tilde P^{ijsl})\big)+o(e^{-(n-1)r})\nonumber\\
&=&\bar\nabla_l\big(V\bar\nabla _j(\zeta_s\tilde P^{ijsl})\big)-2\bar\nabla _j(\zeta_s\tilde P^{ijsl} \bar\nabla_l V)+\bar\nabla_s\big(V\bar\nabla_l(\zeta_j \tilde P^{ijsl})\big)+o(e^{-(n-1)r}).\nonumber
\end{eqnarray}
In the last equality, we have used again $V\in \mathbb{N}_b$. By the first Bianchi identity  \eqref{Bianchi}
of $\tilde P$, we have
$$
\bar\nabla_l\big(V\bar\nabla _j(\zeta_s\tilde P^{ijsl})\big)=-\bar\nabla_l\big(V\bar\nabla _j(\zeta_s\tilde P^{jsil})\big)- \bar\nabla_l\big(V\bar\nabla _j(\zeta_s\tilde P^{sijl})\big).
$$
On the other hand, together with an exchange of a cycle $s,j,l$ and the symmetry, one has
$$
\bar\nabla_l\big(V\bar\nabla _j(\zeta_s\tilde P^{sijl})\big)=\bar\nabla_s\big(V\bar\nabla_l(\zeta_j \tilde P^{ijsl})\big).
$$
Hence we have
\be\label{2form}
\begin{array}{rcl}

\ds \delta\mathbb{U}^i & = & \ds\vs
-\bar\nabla_l\big(V\bar\nabla _j(\zeta_s\tilde P^{jsil})\big)-2\bar\nabla _j(\zeta_s\tilde P^{ijsl} \bar\nabla_lV)+o(e^{-(n-1)r})\\
&=& \ds\vs
 -\bar\nabla_l\big(V\bar\nabla _j(\zeta_s\tilde P^{jsil})+2  \zeta_s\tilde P^{ilsj} \bar\nabla_jV  \big)+o(e^{-(n-1)r})\\
 &=& -\ds \bar\nabla_l \a^{il} +o(e^{-(n-1)r}),
\end{array}\ee
where
\[  \a^{il}= V\bar\nabla _j(\zeta_s\tilde P^{jsil})+2  \zeta_s\tilde P^{ilsj} \bar\nabla_jV, \]
is anti-symmetric.
Therefore we conclude that
$$\mathbb{U}_2^i-\mathbb{U}_1^i=-\ds \bar\nabla_l \a^{il}
+o(e^{-(n-1)r}),$$
that is, up to the parts decay sufficiently fast which will not contribute to the integral at infinity,  $\mathbb{U}_2^i-\mathbb{U}_1^i$ is a divergence of some alternative $2$-vector field. We consider
$\mathbb{U}_2-\mathbb{U}_1$ as a term of  one form. Let $*$ be the Hodge operator with respect to the hyperbolic metric $b$. We have
$$
*(\mathbb{U}_2-\mathbb{U}_1)=d \tilde \alpha+o(e^{-(n-1)r}),
$$
where $\tilde \alpha$ is a $(n-2)$-form obtained from $\alpha$.
This  implies the geometric invariance of the mass functional (see also \cite{CH, Herzlich, M}).
\end{proof}

We now give the  precise  definition of the Gauss-Bonnet-Chern mass \eqref{GBC_h} for asymptotically hyperbolic manifolds in the following definition.
\begin{defi}
If the mass functional $H_k^{\Phi}:\mathbb N_b\rightarrow\R$ is timelike future directed, i.e., $H_k^{\Phi}(V)>0$ for all $V\in \mathbb{N}^+$, then
the higher order mass, the Gauss-Bonnet-Chern mass \eqref{GBC_h} for asymptotically hyperbolic manifolds,  is defined by
\begin{equation}\label{GBC_h2}
 m^{\H}_k:=c(n,k)\inf_{V\in \mathbb{N}_{b}^1}H_k^{\Phi}(V).  \end{equation}
Here $c(n,k)$ is the  normalization constant   given in \eqref{eq_const}.
\end{defi}

The above results show that  $m^{\H}_k$ is an invariant for   asymptotically hyperbolic manifolds.
A similar discussion  as in subsection 2.1 implies
\[ m^{\H}_k= c(n,k)\inf_{\Phi}H_k^{\Phi}(V_{(0)}), \]
where the infimum takes over all asymptotically hyperbolic coordinate $\Phi$ satisfying (\ref{decaytau_order}) and (\ref{add2})
and $V_{(0)}$ is a fixed element in $ \mathbb{N}_{b}^1$.
In the following we fix
$$V=V_{(0)}=\cosh r.$$

We end this section by a discussion on the range of $\tau$.
\begin{rema}\label{rem3.4}
 From (\ref{HV}), one can check directly that  $m_k^{\H}$ vanishes if $\tau>\frac n k$, and hence the well-defined and the non-trivial range for
the Gauss-Bonnet-Chern mass $m_k^\H$ is $\tau\in(\frac{n}{k+1},\frac n k]$. The decay order of the anti-de Sitter Schwarzschild type metric (\ref{metric})  is just $\frac nk.$ Its new mass equals to $m^k$. See  Appendix A for more details.
\end{rema}

\section{Positive mass theorem for asymptotically hyperbolic graphs}

In this section, we investigate a special case that asymptotically hyperbolic
manifolds are given as graphs of asymptotically constant functions over hyperbolic space $\mathbb{H}^n.$
For the new asymptotically hyperbolic mass, we can show that the corresponding Riemannian positive mass theorem holds
for graphs when the modified Gauss-Bonnet curvature is nonnegative.

Following the notation in \cite{DGS}, we identify $\H^{n+1}$ with $(\H^n\times\R, b+V^2ds\otimes ds).$ Let $f:\mathbb H^n \to \mathbb R$ be a smooth function, then the induced metric on the graph
 $$\cM:=\{(x,s)\in\H^n\times\R|f(x)=s\},$$
 is $(\cM^n,g)=(\mathbb{H}^n,b+V^2df\otimes df)$.
 Suppose that the graph $(\cM,g)$ is asymptotically hyperbolic of decay order $\tau>\frac n{k+1}$ and
 $V\tilde L_k(g)$ is integrable.
\begin{rema}
If $f:\mathbb H^n \to \mathbb R$ satisfies
$$\|V\bv f\|_b+\|V\bv^2 f\|_b+ \|V\bv^3 f\|_b=O(e^{-\frac{\tau}{2}r}),\quad \tau>\frac{n}{k+1},$$
then the corresponding graph  $(\cM^n,g)=(\mathbb{H}^n,b+V^2df\otimes df)$ is asymptotically hyperbolic of decay order $\tau>\frac{n}{k+1}$.
\end{rema}
In local coordinates,
\begin{equation}\label{g}
g_{ij}=b_{ij}+V^2\bar\nabla_i f \bar\nabla_j f,
\end{equation}
and the inverse of $g_{ij}$ is
\begin{equation}\label{ginverse}
g^{ij}=b^{ij}-\frac{V^2\bar\nabla^i f \bar\nabla^j f}{1+V^2|\bar\nabla f|^2},
\end{equation}
where the norm is taken with respect to
the hyperbolic metric $b.$

\begin{prop}\label{massP}
Suppose $(\cM^n,g)=(\H^n,b+V^2\bar{\nabla}f\otimes\bar{\nabla}f)$, then we have
\begin{equation}\label{div.Lk}
\bar\nabla_s\bigg(\big(V\bar\nabla_l e_{tj}-e_{tj} \bar\nabla_l V\big) {{\tilde P}_{(k)}}^{stjl}\bigg)=\frac{V}{2}\tilde L_k,
\end{equation}
where $e_{ij}:=g_{ij}-b_{ij}=V^2\bar\nabla_i f \bar\nabla_j f.$
\end{prop}

\begin{proof}
First note that the induced second fundamental form of $\cM^n$ is given by (cf.\cite{DGS})
\begin{equation}\label{II}
h_{ij}=\frac{V}{\sqrt{1+V^2|\bar\nabla f|^2}}\bigg(\bar\nabla_i\bar\nabla_j f+\frac{\bar\nabla_i f\bar\nabla_j V+\bar\nabla_i V\bar\nabla_j f}{V}+V\langle\bar\nabla f, \bar\nabla V\rangle \bar\nabla_i f\bar\nabla_j f\bigg).
\end{equation}
Thus  the shape operator is
\be\label{shape}\begin{array}{rcl}
\ds \vs h^i_j &:=& \ds g^{is}h_{sj}\\
&=&\ds \frac{V}{\sqrt{1+V^2|\bar\nabla f|^2}}\bigg(\bar\nabla^i\bar\nabla_j f+\frac{\bar\nabla^i f\bar\nabla_j V}{V(1+V^2|\bar\nabla f|^2)}+\frac{\bar\nabla^i V\bar\nabla_j f}{V}-\frac{V^2\bar\nabla^i f\bar\nabla^s f\bar\nabla_s\bar\nabla_j f}{1+V^2|\bar\nabla f|^2} \bigg).
\end{array}
\ee
It follows from (\ref{Gammadiff}), (\ref{g}) and (\ref{ginverse}) that
\begin{eqnarray}\label{eq2_Prop4.1}
\Gamma_{ij}^l-\bar{\Gamma}_{ij}^l&=&\frac{\bar{\nabla}^l f}{1+V^2|\bar{\nabla}f|^2}\bigg(V^2\bar{\nabla}_i\bar{\nabla}_j f +V\bar{\nabla}_i V \bar{\nabla}_j f+V\bar{\nabla}_j V \bar{\nabla}_i f\bigg)\nonumber\\
&&-V\bar{\nabla}_i f\bar{\nabla}_j f\bigg(\bar{\nabla}^l V-\frac{V^2\bar{\nabla}^l f\langle\bar{\nabla}f,\bar{\nabla}V\rangle}{1+V^2|\bar{\nabla}f|^2}\bigg).
\end{eqnarray}
In particular,
\begin{equation}\label{eq3_Prop4.1}
{\Gamma}_{is}^s-\bar\Gamma_{is}^s=\frac{\bar{\nabla}^s f}{1+V^2|\bar{\nabla}f|^2}(V^2 \bar{\nabla}_i \bar{\nabla}_s f+ V \bar{\nabla}_i V  \bar{\nabla}_s f).
\end{equation}
Recalling (\ref{tildeRm}), by the Gauss formula we have
\begin{equation}\label{Gauss}
{\tilde R_{ij}}^{\quad\!\!sl}=h_i^s h_j^l-h_i^l h_j^s.
\end{equation}
Substituting (\ref{Gauss}) into (\ref{tildeLk}), (\ref{tildePk}), we infer from (\ref{sigmak}) that
\begin{eqnarray}\label{tildeL&S}
\tilde L_k&=&\d^{i_1i_2\cdots i_{2k-1}i_{2k}}
_{j_1j_2\cdots j_{2k-1}j_{2k}} {h_{i_1}^{j_1}}\cdots {h_{i_{2k}}^{j_{2k}}}\nonumber\\
&=&(2k)!\;\sigma_{2k}(h)
\end{eqnarray}
and
$$\tilde P_{(k)}^{st jl}=\frac 12\d^{i_1i_2\cdots i_{2k-3}i_{2k-2}st}
_{j_1j_2\cdots  j_{2k-3}j_{2k-2} j_{2k-1}j_{2k}} h_{i_1}^{j_1} h_{i_2}^{j_2}\cdots h_{i_{2k-2}}^{j_{2k-2}}g^{j_{2k-1}j}g^{j_{2k}l},
$$
which implies by (\ref{Tk}) that
\begin{equation}\label{relation2}
2{\tilde P_{(k)}}^{stjl} h_{sj}=(2k-1)!\;(T_{(2k-1)})^t_p g^{pl}.
\end{equation}
Here $h=(h_{ij})$ is the second fundamental form given by \eqref{II}.
We derive from (\ref{antisymmetry}),  (\ref{II})  and (\ref{relation2}) that
\begin{eqnarray}\label{eq4_Prop4.1}
\big(V\bar\nabla_l e_{tj}-e_{tj} \bar\nabla_l V\big) {{\tilde P}_{(k)}}^{stjl}
&=&\big(V\bar\nabla_l (V^2\bar{\nabla}_t f\bar{\nabla}_j f )-V^2\bar{\nabla}_t f\bar{\nabla}_j f\bar\nabla_l V\big) {{\tilde P}_{(k)}}^{stjl}\nonumber\\
&=&\big(V^3\bar{\nabla}_l\bar{\nabla}_t f \bar{\nabla}_j f+V^2\bar{\nabla}_l V \bar{\nabla}_t f\bar{\nabla}_j f
\big){{\tilde P}_{(k)}}^{stjl}\nonumber\\
&=&V\left(\bar{\nabla}_l\bar{\nabla}_t f +\frac{\bar{\nabla}_l V \bar{\nabla}_t f}{V}\right){{\tilde P}_{(k)}}^{stjl}(V^2\bar\nabla_j f)\nonumber\\
&=&h_{lt}\sqrt{1+V^2|\bar\nabla f|^2}{{\tilde P}_{(k)}}^{stjl}(V^2\bar\nabla_j f)\nonumber\\
&=&\frac{(2k-1)!}{2}(T_{(2k-1)})_p^s g^{pj}V^2\bar\nabla_j f\sqrt{1+V^2|\bar\nabla f|^2}\nonumber\\
&=&\frac{(2k-1)!}{2}(T_{(2k-1)})_p^s\frac{V^2\bar\nabla^p f}{\sqrt{1+V^2|\bar\nabla f|^2}},
\end{eqnarray}
where the fourth equality follows from (\ref{II}) and (\ref{antisymmetry}).
Since the ambient space is the hyperbolic space $\H^n$ with constant curvature $-1$, the Newton tensor is divergence-free with the induced metric $g$. (For the proof see \cite{ALM} for instance.)

It follows that
\begin{eqnarray*}
&&(\bar\nabla_s(T_{(2k-1)})^s_p) \bar\nabla^p f\\
&=&(\nabla_s(T_{(2k-1)})^s_p) \bar\nabla^p f+(T_{(2k-1)})^s_q(\Gamma_{ps}^q-\bar\Gamma_{ps}^q)\bar\nabla^p f-(T_{(2k-1)})^q_p(\Gamma_{sq}^s-\bar\Gamma_{sq}^s)\bar\nabla^p f\\
&=&(T_{(2k-1)})^s_q(\Gamma_{ps}^q-\bar\Gamma_{ps}^q)\bar\nabla^p f-(T_{(2k-1)})^q_p(\Gamma_{sq}^s-\bar\Gamma_{sq}^s)\bar\nabla^p f.
\end{eqnarray*}
Thus from  (\ref{eq2_Prop4.1}) and (\ref{eq3_Prop4.1}), we have
\begin{eqnarray*}
(\bar\nabla_s(T_{(2k-1)})^s_p)\bar\nabla^p f
&=&(T_{(2k-1)})^s_q\bigg\{\frac{\bar{\nabla}^q f}{1+V^2|\bar{\nabla}f|^2}\left(V^2\bar{\nabla}_p\bar{\nabla}_s f +V\bar{\nabla}_p V \bar{\nabla}_s f+V\bar{\nabla}_s V \bar{\nabla}_p f\right)\\
&&-V\bar{\nabla}_p f\bar{\nabla}_s f\left(\bar{\nabla}^q V-\frac{V^2\bar{\nabla}^q f\langle\bar{\nabla}f,\bar{\nabla}V\rangle}{1+V^2|\bar{\nabla}f|^2}\right)\bigg\}\bar\nabla^p f\\
&&-(T_{(2k-1)})^q_p \frac{\bar{\nabla}^s f}{1+V^2|\bar{\nabla}f|^2}(V^2 \bar{\nabla}_q \bar{\nabla}_s f+ V \bar{\nabla}_q V  \bar{\nabla}_s f)\bar\nabla^p f\\
&=&(T_{(2k-1)})^s_q\left(V\langle\bar\nabla V,\bar\nabla f\rangle\bar\nabla^q f\bar\nabla_s f-V|\bar\nabla f|^2\bar\nabla_s f\bar\nabla^qV\right).
\end{eqnarray*}
Therefore by above,  we derive
\begin{eqnarray*}
&&\bar\nabla_s\left((T_{(2k-1)})^s_p \frac{V^2\bar\nabla^p f}{\sqrt{1+V^2|\bar\nabla f|^2}}\right)\\
&=&\frac{V^2}{\sqrt{1+V^2|\bar\nabla f|^2}}\bar\nabla_s(T_{(2k-1)})^s_p\bar\nabla^p f+(T_{(2k-1)})^s_p\bar\nabla_s\left(\frac{V^2\bar\nabla^p f}{\sqrt{1+V^2|\bar\nabla f|^2}}\right)\\
&=&\frac{V^2}{\sqrt{1+V^2|\bar\nabla f|^2}}(T_{(2k-1)})^s_q\big(V\langle\bar\nabla V,\bar\nabla f\rangle\bar\nabla^q f\bar\nabla_s f-V|\bar\nabla f|^2\bar\nabla_s f\bar\nabla^qV\big)\\
&&+(T_{(2k-1)})^s_p\bigg(\frac{V^2\bar\nabla_s\bar\nabla^p f+2V\bar\nabla_s V\bar\nabla^p f}{\sqrt{1+V^2|\bar\nabla f|^2}}-\frac{V^2\bar\nabla^p f(V|\bar\nabla f|^2\bar\nabla_s V+V^2\bar\nabla_s\bar\nabla^q f\bar\nabla_q f)}{(1+V^2|\bar\nabla f|^2)^{\frac 32}}\bigg)\\
&=&\frac{V(T_{(2k-1)})^s_p}{\sqrt{1+V^2|\bar\nabla f|^2}}\bigg\{\big(V^2\langle\bar\nabla V,\bar\nabla f\rangle\bar\nabla^p f\bar\nabla_s f-V^2|\bar\nabla f|^2\bar\nabla_s f\bar\nabla^pV\big)\\
&&+V\bar\nabla_s\bar\nabla^p f+\bar\nabla_s V\bar\nabla^p f+\frac{\bar\nabla_s V\bar\nabla^p f }{1+V^2|\bar\nabla f|^2}-\frac{V^3\bar\nabla^p f\bar\nabla_s\bar\nabla^q f\bar\nabla_q f}{1+V^2|\bar\nabla f|^2}\bigg\}\\
&=&\frac{V(T_{(2k-1)})^s_p}{\sqrt{1+V^2|\bar\nabla f|^2}}\big(V^2\langle\bar\nabla V,\bar\nabla f\rangle\bar\nabla^p f\bar\nabla_s f-V^2|\bar\nabla f|^2\bar\nabla_s f\bar\nabla^pV+\bar\nabla_s V\bar\nabla^p f-\bar\nabla_s f\bar\nabla^p V\big)\\
&&+V(T_{(2k-1)})^s_p h^p_s.
\end{eqnarray*}
We claim that the first term in the last equality of the above equation indeed vanishes.
We define a new $(1,1)$ tensor by
$$A^p_s:=V^2\langle\bar\nabla V,\bar\nabla f\rangle\bar\nabla^p f\bar\nabla_s f-V^2|\bar\nabla f|^2\bar\nabla_s f\bar\nabla^pV+\bar\nabla_s V\bar\nabla^p f-\bar\nabla_s f\bar\nabla^p V.$$
Then we have
\begin{eqnarray*}
&&\frac{V(T_{(2k-1)})^s_p}{\sqrt{1+V^2|\bar\nabla f|^2}}\big(V^2\langle\bar\nabla V,\bar\nabla f\rangle\bar\nabla^p f\bar\nabla_s f-V^2|\bar\nabla f|^2\bar\nabla_s f\bar\nabla^pV+\bar\nabla_s V\bar\nabla^p f-\bar\nabla_s f\bar\nabla^p V\big)\\
&=& \frac{V(T_{(2k-1)})^s_p}{\sqrt{1+V^2|\bar\nabla f|^2}}A_s^p
= \frac{V(T_{(2k-1)})^{sq}}{\sqrt{1+V^2|\bar\nabla f|^2}}g_{pq} A_s^p.
\end{eqnarray*}
We only need to show that
$g_{pq}A_s^p$ is anti-symmetric, which is in fact easy to check
\begin{eqnarray*}
g_{pq}A_s^p&=&(b_{ pq}+V^2\bv_p f\bv _q f)A^p_s \\
&=& V^2 \langle \bv V, \bv f\rangle \bv _q f\bv_s f-V^2 |\bv f|^2 \bv_s f \bv _q V + \bv _s V \bv _q f -\bv_s f \bv _q V \\
&& +V^4 \langle \bv V, \bv f\rangle |\bv f|^2 \bv_q f\bv_s f -V^4 |\bv f|^2 \langle \bv f,\bv V\rangle \bv_s f \bv _q f \\
&&+V^2 |\bv f|^2 \bv_s V \bv_q f - V^2\langle \bv f, \bv V \rangle \bv_q f \bv_s f\\
&=&(1+V^2|\bar\nabla f|^2)(\bv _s V \bv _q f -\bv_s f \bv _q V).
\end{eqnarray*}
Since $(T_{(2k-1)})^{sq}$ is symmetric, we get the desired result that
$$\frac{V(T_{(2k-1)})^s_p}{\sqrt{1+V^2|\bar\nabla f|^2}}\big(V^2\langle\bar\nabla V,\bar\nabla f\rangle\bar\nabla^p f\bar\nabla_s f-V^2|\bar\nabla f|^2\bar\nabla_s f\bar\nabla^pV+\bar\nabla_s V\bar\nabla^p f-\bar\nabla_s f\bar\nabla^p V\big)=0,$$
which yields from (\ref{sigmak})  that
\begin{equation}\label{eq4.12}
\bar\nabla_s\left((T_{(2k-1)})^s_p \frac{V^2\bar\nabla^p f}{\sqrt{1+V^2|\bar\nabla f|^2}}\right)=V(T_{(2k-1)})^s_p h^p_s=2k V\sigma_{2k}(h).
\end{equation}
Finally,
\eqref{div.Lk} follows from \eqref{tildeL&S}, (\ref{eq4_Prop4.1}) and \eqref{eq4.12}.
\end{proof}

\begin{lemm}\label{lem_rem}
 From (\ref{eq4_Prop4.1}), we  have the following equivalent form of hyperbolic mass integral (\ref{HV}) for asymptotically hyperbolic graphs
\begin{equation}\label{equv.HV}
m_k^{\H}=c(n,k)\frac{(2k-1)!}{2}\lim_{r\rightarrow\infty}\int_{S_r}(T_{(2k-1)})_p^s\frac{V^2\bar\nabla^p f}{\sqrt{1+V^2|\bar\nabla f|^2}}\nu_s d\mu.
\end{equation}
\end{lemm}

\begin{rema}
The method to show  Proposition $\ref{massP}$ is motivated by a classical paper of Reilly  \cite{Reilly0}. The method gives a simpler proof even for $k=1$ case (cf. \cite{DGS}). A similar argument works for the Gauss-Bonnet-Chern mass for
asymptotically flat manifolds, which provides a simply proof for the corresponding results in \cite{GWW}. Moreover, it also
provides a brief proof of the corresponding  Penrose inequality. See the next Section.

\end{rema}

\vspace{3mm}
Now we are ready to prove our main Theorem \ref{PMthm}.

\

\noindent {\it Proof of  Theorem \ref{PMthm}.}
Applying divergence theorem with Proposition \ref{massP}, we have
\begin{eqnarray*}
m_k^{\H}&=&c(n,k)\int_{\cM}\bar\nabla_s\bigg(\big(V\bar\nabla_l e_{tj}-e_{tj} \bar\nabla_l V\big) {{\tilde P}_{(k)}}^{stjl}\bigg)dV_b\\
&=&\frac{c(n,k)}{2}\int_{\cM}V\tilde L_k dV_b\\
&=&\frac{c(n,k)}{2}\int_{\cM}\frac{V\tilde L_k}{\sqrt{1+V^2|\bar\nabla f|^2}} dV_g,
\end{eqnarray*}
where the last equality holds due to the fact
$$dV_g=\sqrt{\mbox{det}g}\;dV_{b}=\sqrt{1+V^2|\bar\nabla f|^2}\;dV_{b}.$$
\qed

\section{Penrose inequality for graphs over $\mathbb{H}^n$ with a horizon type boundary }
Now we investigate the Penrose inequality related to the new mass for the asymptotically hyperbolic manifolds which can be realized as graphs.
Let $\Omega$ be a bounded open set in $\mathbb{H}^n$ and $\Sigma=\partial\Omega.$  If $f:\mathbb{H}^n\setminus\Omega\rightarrow \mathbb{R}$  is a smooth function
such that each connected component of $\Sigma$ is in a level set of $f$ and
\begin{equation}\label{x1}
|\bar\nabla f(x)|\rightarrow\infty \quad\hbox{  as }
x\rightarrow\Sigma,\end{equation}
then the graph of $f$, $(\cM^n,g)=(\mathbb{H}^n\setminus\Omega,b+V^2df\otimes df),$ is an asymptotically hyperbolic manifold with a
horizon $\Sigma.$  Without loss of generality we may assume that $\{(x, f(x) )\,|\, x\in \Sigma\}$  is included in $f^{-1}(0)$. In this case we can
identify $\{(x, f(x)) \,|\, x\in \Sigma\}$ with $\Sigma$.

On $\Sigma$, the outer unit normal vector induced by the hyperbolic metric $b$ is $$ \nu:=  \nu^i\frac{\partial}{\partial x^i }=-\frac{\bar\nabla
f}{|\bar\nabla f|}.$$
Then
$$\nu^i=-\frac{b^{ij}\bar{\nabla}_j f}{|\bar{\nabla} f|}=-\frac{\bar{\nabla}^i f}{|\bar{\nabla} f|}\quad\mbox{and}\quad \nu_i:= b_{ij}\nu^j=-\frac{\bar{\nabla}_i f}{|\bar{\nabla} f|}.$$

\begin{rema} \label{rem5.1} One can easily check  that the second fundamental forms of $\Sigma$ induced by $g$ and $b$ respectively differ by a multiple $
\frac{1}{\sqrt{1+V^2|\bv f|^2}}$. Hence we have the following equivalent statements, provided $\Sigma\subset\mathbb{H}^n$ is strictly mean convex:
\begin{itemize} \item $|\bar\nabla f|=\infty$ on $\Sigma$;
\item $\Sigma$ is minimal, i.e., the induced mean curvature by the metric $g$ vanishes;
\item $\Sigma$ is totally geodesic, i.e., the induced second fundamental form by the metric $g$  vanishes.
\end{itemize}
Therefore, in this case  $\Sigma$ is an area-minimizing  horizon if and only if $|\bar\nabla f|=\infty$ on $\Sigma$.
Hence
$|\bar\nabla f|=\infty$  is a natural assumption for horizons.
\end{rema}

\

\noindent{\it Proof of Theorem \ref{re}.}
In view of (\ref{equv.HV}), integrating by parts now gives an extra boundary term,
\begin{eqnarray*}
\frac{1}{c(n,k)}m_k^{\H}&=&\frac{(2k-1)!}{2}\lim_{r\rightarrow\infty}\int_{S_r}(T_{(2k-1)})_p^s\frac{V^2\bar\nabla^p f}{\sqrt{1+V^2|\bar\nabla f|^2}}\nu_s d\mu\\
&=&\frac{1}{2}\int_{\mathbb{H}^n\setminus\Omega}\frac{V\tilde L_k}{\sqrt{1+V^2|\bar\nabla f|^2}}dV_g-\frac{(2k-1)!}{2}\int_{\Sigma}(T_{(2k-1)})^s_p\frac{V^2\bv^p f}{\sqrt{1+V^2|\bar\nabla f|^2}}\nu_sd\mu.
\end{eqnarray*}
We  may choose the  coordinates such that $\{\frac{\partial}{\partial{x^2}},\cdots,\frac{\partial}{\partial{x^n}}\}$ span the tangential space
 of $\Sigma$ and $\frac{\partial}{\partial{x^1}}$ denotes the normal direction of $\Sigma$. To clarify the notations, in the following
  we will use the convention that the Latin letters stand for the index: $1,2,\cdots, n$
and the Greek letters stand for the index: $2, \cdots, n$.
 Due to the  assumption that  $\Sigma$ is in a level set of $f$, at any given point $p\in\Sigma$, we have
 \begin{equation}\label{h&S}
 \bv_{\alpha} f=0 \quad\mbox{and}\quad \bv_{\alpha}\bv_{\beta}f=|\bv f|B_{\alpha\beta},
 \end{equation}
 where $B$ is the second fundamental form with respect to the inward normal vector of $(\Sigma,\gamma)$ induced by the metric $b$.
Therefore we infer from (\ref{sigmak}) and (\ref{Tk}) that
\begin{eqnarray*}
\frac{1}{c(n,k)}m^{\H}_k&=&\frac{1}{2}\int_{\mathbb{H}^n\setminus\Omega}\frac{V\tilde L_k}{\sqrt{1+V^2|\bar\nabla f|^2}}dV_g+\frac{(2k-1)!}{2}\int_{\Sigma}(T_{(2k-1)})^1_1\frac{V^2|\bv f|}{\sqrt{1+V^2|\bar\nabla f|^2}}d\mu\\
&=&\frac{1}{2}\int_{\cM^n}\frac{V\tilde L_k}{\sqrt{1+V^2|\bar\nabla f|^2}}dV_g+\frac{(2k-1)!}{2}\int_{\Sigma}\sigma_{2k-1}V\bigg( \frac{V^2|\bar\nabla f|^2}{1+V^2|\bar\nabla f|^2}\bigg)^kd\mu\\
&=&\frac{1}{2}\int_{\cM^n}\frac{V\tilde L_k}{\sqrt{1+V^2|\bar\nabla f|^2}}dV_g+\frac{(2k-1)!}{2}\int_{\Sigma}V\sigma_{2k-1}d\mu.
\end{eqnarray*}
Here we have used the simple fact  in the  second equality
$$ (T_{(2k-1)})^1_1=\frac{1}{(2k-1)!} \delta_{1i_1i_2\cdots i_{2k-1}}^{1j_1j_2\cdots j_{2k-1}}h_{j_1}^{i_1} \cdots h_{j_{2k-1}}^{i_{2k-1}}
=\bigg( \frac{V|\bar\nabla f|}{\sqrt{1+V^2|\bar\nabla f|^2}}\bigg)^{2k-1}\s_{2k-1},$$
which follows from $$h^\alpha_\beta=\frac{V|\bar\nabla f|}{\sqrt{1+V^2|\bar\nabla f|^2}} b^{\alpha\delta}B_{\delta\beta}=\frac{V|\bar\nabla f|}{\sqrt{1+V^2|\bar\nabla f|^2}}B^{\alpha}_{\beta},$$ by (\ref{shape}) and (\ref{h&S}) (where $B^{\alpha}_{\beta}:=\gamma^{\alpha\delta}B_{\delta\beta}=
b^{\alpha\delta}B_{\delta\beta}$ )
and the last equality holds by the assumption (\ref{x1}). This finishes the proof of the Theorem. \qed

Using the Alexandrov-Fenchel inequality (\ref{eq1_thm}) proved in  the second part below, we have the following Penrose inequality, which is slightly stronger than
Theorem \ref{thm_P}.

\begin{theo}[Penrose Inequality]\label{Penrose thm}
Let $\Omega$ be a bounded open set in $\mathbb{H}^n$ and $\Sigma=\partial\Omega.$ If $f:\mathbb{H}^n\setminus\Omega\rightarrow \mathbb{R}$ is a smooth
function such that the graph $(\cM^n,g)=(\H^n\setminus\Omega,b+V^2df\otimes df)$ is asymptotically hyperbolic of decay order $\tau>\frac n{k+1}$ and $V\tilde L_k$ is integrable. Assume that
 each connected component of $\Sigma$ is in a level set of $f$ and $|\bar\nabla f(x)|\rightarrow\infty$ as
$x\rightarrow\Sigma$. Let $\Omega_i$ be connected components of $\Omega,i=1,\cdots,l$ and
let $\Sigma_i=\partial\Omega_i$ and suppose that each $\Omega_i$ is horospherical
convex, then
$$m^{\H}_k \ge  c(n,k)\int_{\cM^n}\frac 12\frac{V\tilde L_k}{\sqrt{1+V^2|\bv f|^2}}dV_g+\sum_{i=1}^l\frac{1}{2^k}{\left(\left(\frac{|\Sigma_i|}{\omega_{n-1}}\right)^{\frac{n}{k(n-1)}}+\left(\frac{|\Sigma_i|}{\omega_{n-1}}\right)^{\frac{n-2k}{k(n-1)}} \right)}^{k}.$$
In particular, $\tilde L_k\geq0$ implies
$$m^{\H}_k\geq\sum_{i=1}^l\frac{1}{2^k}{\left(\left(\frac{|\Sigma_i|}{\omega_{n-1}}\right)^{\frac{n}{k(n-1)}}+\left(\frac{|\Sigma_i|}{\omega_{n-1}}\right)^{\frac{n-2k}{k(n-1)}} \right)}^{k}.$$
Moreover,  equality is achieved by the anti-de Sitter Schwarzschild type metric (\ref{metric}).
\end{theo}

\begin{proof}The Penrose inequality follows from Theorem \ref{re} and the  Alexandrov-Fenchel inequality, Theorem \ref{wAF} below.
The last statement is proved in the following example.
\end{proof}

We end this section with an interesting example.

\begin{exam}\label{Schwarzchild}
The generalized anti-de Sitter Schwarzschild space-time is given by
\beq\label{metric_SR}
-(1+\rho^2-\frac{2m}{\rho^{\frac nk-2}})dt^2+(1+\rho^2-\frac{2m}{\rho^{\frac nk-2}})^{-1}d\rho^2+\rho^2d\Theta^2,
\eeq
where $d\Theta^2$ is the round metric on $\mathbb{S}^{n-1}.$
When $k=1$ we recover anti-de Sitter Schwarzschild space-time of the Einstein gravity. \\
\end{exam}
See  also \cite{CTZ} or example $2.8$ in \cite{CLS}.
Restricting to the time slice $t=0$, we obtain the anti-de Sitter Schwarzschild metric
\begin{equation}\label{metric1}
g_{\rm adS-Sch}=(1+\rho^2-\frac{2m}{\rho^{\frac nk-2}})^{-1}d\rho^2+\rho^2d\Theta^2.
\end{equation}
One can realizes  metric (\ref{metric1})   as a graph over the hyperbolic metric $\H^{n}.$
 In the transformation of coordinates $\rho=\sinh r$, the hyperbolic metric $b=dr^2+\sinh^2 rd\Theta^2$ can be rewritten as $$b=\frac{d\rho^2}{1+\rho^2}+\rho^2d\Theta^2.$$
Hence to explicitly express the metric (\ref{metric1}) as a graph over $\H^{n}$, we need to find a function $f=f(\rho)$ satisfying
 $$V^2\left(\frac{\partial f}{\partial\rho}\right)^2=\frac{1}{1+\rho^2-\frac{2m}{\rho^{\frac n k-2}}}-\frac{1}{1+\rho^2}.$$
 Note that the function V  is now  equal to $\sqrt{1+\rho^2}$.   Let $\rho_0$ be the solution of
$$1+\rho^2-\frac{2m}{\rho^{\frac nk-2}}=0.$$ Then when $\rho$ approaches $\rho_0$, we have $\frac{\partial f}{\partial\rho}=O((\rho-\rho_0)^{-\frac 12}),$
 so that we can solve
 $$f(\rho)=\int_{\rho_0}^{\rho}\frac{1}{\sqrt{1+s^2}}\;\sqrt{\frac{1}{1+s^2-\frac{2m}{s^{\frac n k-2}}}-\frac{1}{1+s^2}}\;\;ds.$$
One may compare with the Euclidean case where the Schwarzschild metric can be written as a graph over $\R^n$ \cite{Lam}.
This example is a generalization of the one considered in \cite{DGS}, where is in fact the case $k=1$.

One can check  that the metric $g=g_{{\rm adS-Sch}}$  satisfies
\begin{equation}\label{Lk-sch}
\tilde L_k(g)=0,
\end{equation}
and
\begin{equation}\label{mass-sch}
m_k^{\H}=m^k.
\end{equation}
For the convenience of  the reader, we include the computations of (\ref{Lk-sch}) and (\ref{mass-sch}) in Appendix A.

In this example, the horizon is given by the surface $\rho=\rho_0$ with $\rho_0$ being the solution of
$$1+\rho^2-\frac{2m}{\rho^{\frac nk-2}}=0.$$
Then the horizon is $\{S_{\rho_0}:\rho_0^{\frac nk}+\rho_0^{\frac nk-2}=2m\}$ which implies the right hand side of inequality  (\ref{Pen_k}) is
\begin{eqnarray*}
RHS&=&\frac 1{2^k}\left(\left(\frac{\omega_{n-1}\rho_0^{n-1}}{\omega_{n-1}}\right)^{\frac{n}{k(n-1)}}+\left(\frac{\omega_{n-1}\rho_0^{n-1}}{\omega_{n-1}}\right)^{\frac{n-2k}{k(n-1)}}\right)^k\\
&=&\frac 1{2^k}\left(\rho_0^{\frac nk}+\rho_0^{\frac nk-2}\right)^k=\frac 1{2^k}(2m)^k\\
&=&m^k=m_k^{\H}.
\end{eqnarray*}
This means that equality in (\ref{Pen_k}) is  achieved by the adS Schwarzschild type metric \eqref{metric}.  \qed

{\appendix

\section{The anti-de Sitter Schwarzschild type metric}
In this Appendix, we first compute the modified Gauss-Bonnet curvature $\tilde L_k$ for the following Riemannian metric which was discussed in Section 5
\begin{equation}\label{Sch.app.}
g_{\rm adS-Sch}=(1+\rho^2-\frac{2m}{\rho^{\frac nk-2}})^{-1}d\rho^2+\rho^2d\Theta^2.
\end{equation}
Let us denote
$$
\varphi(\rho)= \sqrt{1+\rho^2-\frac{2m}{\rho^{\frac nk-2}}}.
$$

We have  for any vectors $X$, $Y$ tangential to the sphere $S_{\rho}$
\begin{eqnarray}\label{normaldir.}
\mathscr{R}(X\wedge\partial \rho)+X\wedge\partial \rho&=&-\left(\frac{n}{2k}-1\right)\frac{2m}{\rho^{\frac nk}}X\wedge\partial \rho,\\
\mathscr{R}(X\wedge Y)+X\wedge Y&=&\frac{2m}{\rho^{\frac nk}}X\wedge Y,\label{tang.dir.}
\end{eqnarray}
where $\mathscr{R}$ is the curvature operator.
This follows from a directly computation.
Inserting  \eqref{normaldir.}  and \eqref{tang.dir.}
into the (\ref{tildeLk}),  one can check that
$$\tilde L_k=(2k)!\left(C_{n-1}^{2k}-C_{n-1}^{2k-1}\left(\frac{n}{2k}-1\right)\right)\left(\frac{2m}{\rho^{\frac nk}}\right)^{k}=0.$$

\

In the following, we compute the new mass of metric (\ref{Sch.app.}) . Let $\frac{\partial}{\partial x^1}$ denotes the $\frac{\partial}{\partial \rho}$ direction and $\{\frac{\partial}{\partial x^2},\cdots\frac{\partial}{\partial x^n}\}$ be an orthogonal basis on $\S^{n-1}.$ We will use the convention that the Latin letters stand for the
index: $1, 2, \cdots, n$ and the Greek letters stand for the index:
$2, \cdots, n$.
Since  the hyperbolic metric  $$b=\frac{1}{1+\rho^2}d\rho^2+\rho^2d\Theta^2,$$ the only non-vanishing term is
$$e_{11}=g_{11}-b_{11}=\frac{1}{1+\rho^2-\frac{2m}{\rho^{\frac nk-2}}}-\frac{1}{1+\rho^2}=\frac{2m}{\rho^{\frac{n}{k}+2}}+\mbox{lower order terms}.$$
On the other hand, from (\ref{normaldir.}) and (\ref{tang.dir.}), we know that for the curvature term $\tilde R_{i_1 i_2}^{\quad\;j_1 j_2}$, the indices $\{j_1,j_2\}$ should be a permutation of $\{i_1,i_2\}$, otherwise $\tilde R_{i_1 i_2}^{\quad\; j_1 j_2}$ vanishes. Consequently, by  the definition of $\tilde P^{stlm}_{(k)}$ we compute by using  (\ref{normaldir.}), (\ref{tang.dir.}) that for fixed $\alpha,\beta,$
\begin{eqnarray}\label{eq1_appen.}
{{\tilde P}_{(k)}}^{1\alpha1\beta}
&=&\frac{1}{2^k}\d^{i_1i_2\cdots i_{2k-3}i_{2k-2}1\alpha}
_{j_1j_2\cdots  j_{2k-3}j_{2k-2} j_{2k-1}j_{2k}}{\tilde R_{i_1i_2}}^{\quad\, j_1 j_2}\cdots
{\tilde R_{i_{2k-3}i_{2k-2}}}^{\qquad\quad j_{2k-3}j_{2k-2}}g^{j_{2k-1}1}g^{j_{2k}\beta}\nonumber\\
&=&\frac{1}{2^k}\d^{i_1i_2\cdots i_{2k-3}i_{2k-2}1\alpha}
_{j_1j_2\cdots  j_{2k-3}j_{2k-2} 1\alpha}{\tilde R_{i_1i_2}}^{\quad\, j_1 j_2}\cdots
{\tilde R_{i_{2k-3}i_{2k-2}}}^{\qquad\quad j_{2k-3}j_{2k-2}}g^{11}g^{\alpha\beta}\nonumber\\
&=&\frac{1}{2^k} 2^{k-1}(2k-2)! \;C_{n-2}^{2k-2}\left(\frac{2m}{\rho^{\frac nk}}\right)^{k-1}g^{11}g^{\alpha\beta}\nonumber\\
&=&2^{k-2}(2k-2)!\; C_{n-2}^{2k-2}\left(\frac{m}{\rho^{\frac nk}}\right)^{k-1}(1+\rho^2-\frac{2m}{\rho^{\frac  nk-2}})g^{\alpha\beta}.
\end{eqnarray}
Moreover, from above calculations, one can check that this kind of term ${{\tilde P}_{(k)}}^{1\alpha\beta\delta}=0.$
And a direct computation gives   $$\bar\nabla_t(e_{\alpha\beta})=0.$$
Noticing the fact $$\nu_{\alpha}=0,$$ we derive from above that
\begin{eqnarray*}
(V\bar\nabla_t(e_{js})-e_{js}\bar\nabla_t V)\tilde P^{ijst}_{(k)}\nu_i
&=&V\bar\nabla_{\beta}(e_{1\alpha})\tilde P_{(k)}^{1\alpha1\beta}\nu_1.
\end{eqnarray*}
Using
$$
\bar\Gamma_{\alpha\beta}^1=-\rho(1+\rho^2)\delta_{\alpha\beta},
$$
and recalling $V=\sqrt{1+\rho^2}$, we compute further from above that
\begin{eqnarray}\label{eq2_appen.}
(V\bar\nabla_t(e_{js})-e_{js}\bar\nabla_t V)\tilde P^{ijst}_{(k)}\nu_i
=\frac{2m}{\rho^{\frac nk-2}}\delta_{\alpha\beta}\tilde P_{(k)}^{1\alpha1\beta}\nu_1+\mbox{lower order terms}.
\end{eqnarray}
Hence we infer from (\ref{eq1_appen.}) and (\ref{eq2_appen.}) that
\begin{eqnarray*}
(V\bar\nabla_t(e_{js})-e_{js}\bar\nabla_t V)\tilde P^{ijst}_{(k)}\nu_i&=&2^{k-1}(2k-2)!\; C_{n-2}^{2k-2}\frac{m^k}{\rho^{n-2}}(n-1)\nu_1+\mbox{lower order terms}\\
&=&2^{k-1}(2k-2)!\; C_{n-2}^{2k-2}\frac{m^k}{\rho^{n-1}}(n-1)+\mbox{lower order terms},
\end{eqnarray*}
where in the second equality we have used the fact $\nu_1=\frac 1{\sqrt{1+\rho^2}}.$
Therefore we conclude
\begin{eqnarray*}
m_k^{\H}&=&\frac{(n-2k)!} {2^{k-1}(n-1)!\;\omega_{n-1}}\lim_{\rho\rightarrow\infty}\int_{S_{\rho}}(V\bar\nabla_t(e_{js})-e_{js}\bar\nabla_t V)\tilde P^{ijst}_{(k)}\nu_i d\mu\\
&=&\frac{(n-2k)!} {2^{k-1}(n-1)!\;\omega_{n-1}} 2^{k-1}(2k-2)!\; C_{n-2}^{2k-2}\frac{m^k}{\rho^{n-1}}(n-1)\rho^{n-1}\omega_{n-1}\\
&=&m^k.
\end{eqnarray*}

\section{Proof of (\ref{ast})}
In this Appendix, we give the proof of (\ref{ast}).

\

\noindent{\it Proof of \eqref{ast}.}
Denote $g_1=({\Phi_1}^{-1})^{\ast}g$ and $g_2=({\Phi_2}^{-1})^{\ast}g$. Recall from \eqref{linearize} that we have
\be
\label{y1}
g_2 -g_1= {\mathcal L}_\zeta b+o(e^{-\frac{2n}{k+1}r}), \quad  \hbox{ with\;  } \zeta=o(e^{-\frac{n}{k+1}r}).
\ee
It follows easily that
\begin{equation}\label{fact_appen.}
\tilde R_{ij}^{\quad\!\! sp}(g_1)=o(e^{-\frac{n}{k+1}r}),\;\;\quad \tilde R_{ij}^{\quad\!\! sp}(g_2)=o(e^{-\frac{n}{k+1}r}).
\end{equation}
In order to prove (\ref{ast}), in view of (\ref{y1}) and (\ref{tildePk}), it suffices to show that
\begin{equation}\label{aim_appen.}
\tilde R_{ij}^{\quad\!\!sp}(g_2)-\tilde R_{ij}^{\quad\!\! sp}(g_1)=o(e^{-\frac{2n}{k+1}r}).
\end{equation}
In the proof, what we need to take care of is why the linear terms involving $\zeta$ cancel.

First by (\ref{tildeRm}), we compute
\begin{eqnarray*}
&&\tilde R_{ij}^{\quad\!\! sp}(g_2)-\tilde R_{ij}^{\quad\!\! sp}(g_1)\\
&=&({R_{ij}}^{sp}(g_2)+\delta_i^s\delta_j^p-\delta_i^p\delta_j^s)-({R_{ij}}^{sp}(g_1)+\delta_i^s\delta_j^p-\delta_i^p\delta_j^s)\\
&=&g_2^{pq}R_{ij\;\;\; q}^{\quad\! s}(g_2)-g_1^{pq}R_{ij\;\;\; q}^{\quad\! s}(g_1)=(g_2^{pq}-g_1^{pq})R_{ij\;\;\; q}^{\quad\! s}(g_2)+g_1^{pq}(R_{ij\;\;\; q}^{\quad\! s}(g_2)-R_{ij\;\;\; q}^{\quad\! s}(g_1))\\
&=&-({\mathcal L}_\zeta b)^{pq}\left(-(\delta_i^s(g_2)_{jq}-\delta_j^s(g_2)_{iq})\right)+g_1^{pq}(R_{ij\;\;\; q}^{\quad\! s}(g_2)-R_{ij\;\;\; q}^{\quad\! s}(g_1))+o(e^{-\frac{2n}{k+1}r})\\
&=:&I+II+o(e^{-\frac{2n}{k+1}r}),
\end{eqnarray*}
where in the third  equality, we have used the fact (\ref{fact_appen.}) and expression (\ref{y1}).

We begin with the simpler term $I$. We have from (\ref{y1}) that
\begin{eqnarray}\label{I_appen.}
I&:=&-({\mathcal L}_\zeta b)^{pq}\left(-(\delta_i^s(g_2)_{jq}-\delta_j^s(g_2)_{iq})\right)\nonumber\\
&=&-(\bv^p\zeta^q+\bv^q\zeta^p)\left(-(\delta_i^sb_{jq}-\delta_j^sb_{iq})\right)+o(e^{-\frac{2n}{k+1}r})\nonumber\\
&=&\delta_i^s(\bv^p\zeta_j+\bv_j\zeta^p)-\delta_j^s(\bv^p\zeta_i+\bv_i\zeta^p)+o(e^{-\frac{2n}{k+1}r}).
\end{eqnarray}
In the following, we deal with the second term $II$.  In view of (\ref{Gammadiff}), we have
\begin{eqnarray*}
\Gamma_{ij}^s(g_2)-\Gamma_{ij}^s(g_1)&=&\frac 12 (g_2)^{sl}\left(\nabla^{g_1}_i(g_2-g_1)_{lj}+\nabla^{g_1}_j(g_2-g_1)_{li}-\nabla^{g_1}_l(g_2-g_1)_{ij}\right)\\
&=&\frac 12 b^{sl}\left(\bv_i(g_2-g_1)_{lj}+\bv_j(g_2-g_1)_{li}-\bv_l(g_2-g_1)_{ij}\right)+o(e^{-\frac{2n}{k+1}r}).
\end{eqnarray*}
Substituting above into (\ref{Rmdiff}) and noting that the quadratic terms of Christoffel symbols having faster decay, we calculate
\begin{eqnarray*}
R_{ij\;\;\; q}^{\quad\! s}(g_2)-R_{ij\;\;\; q}^{\quad\! s}(g_1)&=&\bv_i(\Gamma_{jq}^s(g_2)-\Gamma_{jq}^s(g_1))-\bv_j(\Gamma_{iq}^s(g_2)-\Gamma_{iq}^s(g_1))+o(e^{-\frac{2n}{k+1}r})\\
&=&\frac 12 b^{sl}\big(\bv_i\bv_j(g_2-g_1)_{lq}+\bv_i\bv_q(g_2-g_1)_{lj}-\bv_i\bv_l(g_2-g_1)_{jq}\\
&&-\bv_j\bv_i(g_2-g_1)_{lq}-\bv_j\bv_q(g_2-g_1)_{li}+\bv_j\bv_l(g_2-g_1)_{iq}\big)+o(e^{-\frac{2n}{k+1}r})\\
&=&\frac 12 b^{sl}\big[(\bv_i\bv_j(g_2-g_1)_{lq}-\bv_j\bv_i(g_2-g_1)_{lq})+(\bv_i\bv_q\bv_j\zeta_l-\bv_j\bv_q\bv_i\zeta_l)\\
&&+(\bv_i\bv_q\bv_l\zeta_j-\bv_i\bv_l\bv_q\zeta_j)+(\bv_j\bv_l\bv_i\zeta_q-\bv_i\bv_l\bv_j\zeta_q)\\
&&+(\bv_j\bv_l\bv_q\zeta_i-\bv_j\bv_q\bv_l\zeta_i)]+o(e^{-\frac{2n}{k+1}r})\\
&=&\frac 12 b^{sl}(II_1+II_2+II_3+II_4+II_5)+o(e^{-\frac{2n}{k+1}r}),
\end{eqnarray*}
where in the third equality, we have used (\ref{y1}) and rearranged these terms.

In the following computations, we mainly rely on the Ricci identity to calculate each term carefully.
Applying the Ricci identity, we have
\begin{eqnarray*}
II_1&:=&\bv_i\bv_j(g_2-g_1)_{lq}-\bv_j\bv_i(g_2-g_1)_{lq}\\
&=&(g_2-g_1)_{lt}(-\bar R_{ij\;\;q}^{\quad\!t})+(g_2-g_1)_{tq}(-\bar R_{ij\;\;\,l}^{\quad\!t})\\
&=&(g_2-g_1)_{lt}(\delta_i^t b_{jq}-\delta_j^tb_{iq})+(g_2-g_1)_{tq}(\delta_i^t b_{jl}-\delta_j^tb_{il})\\
&=&(g_2-g_1)_{il}b_{jq}-(g_2-g_1)_{lj}b_{iq}+(g_2-g_1)_{iq}b_{jl}-(g_2-g_1)_{jq}b_{il}.
\end{eqnarray*}
Using the Ricci identity three times, we obtain,
\begin{eqnarray*}
II_2&:=&\bv_i\bv_q\bv_j\zeta_l-\bv_j\bv_q\bv_i\zeta_l\\
&=&(\bv_i\bv_q\bv_j\zeta_l-\bv_q\bv_i\bv_j\zeta_l)+(\bv_q\bv_i\bv_j\zeta_l-\bv_q\bv_j\bv_i\zeta_l)+(\bv_q\bv_j\bv_i\zeta_l-\bv_j\bv_q\bv_i\zeta_l)\\
&=&\bv_j\zeta_t(-\bar R_{iq\;\,l}^{\quad\!\! t})+\bv_t\zeta_l(-\bar R_{iq\;\,j}^{\quad\!\!t})+\bv_q\zeta_t(-\bar R_{ij\;\;\,l}^{\quad\!\!t})+\bv_t\zeta_l(-\bar R_{qj\;\,\,i}^{\quad\!\! t})+\bv_i\zeta_t(-\bar R_{qj\;\,\,l}^{\quad\!\! t})\\
&=&b_{ql}(\bv_j\zeta_i-\bv_i\zeta_j)-b_{il}(\bv_q\zeta_j+\bv_j\zeta_q)+b_{jl}(\bv_q\zeta_i+\bv_i\zeta_q)+b_{qj}\bv_i\zeta_l-b_{qi}\bv_j\zeta_l.
\end{eqnarray*}
Similarly,  we get
\begin{eqnarray*}
II_3&:=&\bv_i\bv_q\bv_l\zeta_j-\bv_i\bv_l\bv_q\zeta_j=-b_{qj}\bv_i\zeta_l+b_{lj}\bv_i\zeta_q,\\
\vspace{3mm}
II_4&:=&\bv_j\bv_l\bv_i\zeta_q-\bv_i\bv_l\bv_j\zeta_q\\
&=&b_{ql}(\bv_i\zeta_j-\bv_j\zeta_i)-b_{jq}(\bv_l\zeta_i+\bv_i\zeta_l)+b_{iq}(\bv_l\zeta_j+\bv_j\zeta_l)+b_{li}\bv_j\zeta_q-b_{lj}\bv_i\zeta_q,\\
\vspace{3mm}
II_5&:=&\bv_j\bv_l\bv_q\zeta_i-\bv_j\bv_q\bv_l\zeta_i=-b_{li}\bv_j\zeta_q+b_{qi}\bv_j\zeta_l.
\end{eqnarray*}
By a simplification, we infer by (\ref{y1}) that
\begin{eqnarray*}
&&II_1+(II_2+II_3+II_4+II_5)\\
&=&(g_2-g_1)_{il}b_{jq}-(g_2-g_1)_{lj}b_{iq}+(g_2-g_1)_{iq}b_{jl}-(g_2-g_1)_{jq}b_{il}\\
&&-b_{jq}(\bv_i\zeta_l+\bv_l\zeta_i)+b_{jl}(\bv_q\zeta_i+\bv_i\zeta_q)-b_{il}(\bv_j\zeta_q+\bv_q\zeta_j)+b_{iq}(\bv_l\zeta_j+\bv_j\zeta_l)\\
&=&-2b_{il}(\bv_j\zeta_q+\bv_q\zeta_j)+2b_{jl}(\bv_i\zeta_q+\bv_q\zeta_i)+o(e^{-\frac{2n}{k+1}r}).
\end{eqnarray*}
Hence we conclude
\begin{eqnarray}\label{II_appen.}
II&:=&g_1^{pq}(R_{ij\;\;\; q}^{\quad\! s}(g_2)-R_{ij\;\;\; q}^{\quad\! s}(g_1))\nonumber\\
&=&g_1^{pq}(\frac 12 b^{sl})\left(-2b_{il}(\bv_j\zeta_q+\bv_q\zeta_j)+2b_{jl}(\bv_i\zeta_q+\bv_q\zeta_i)\right)+o(e^{-\frac{2n}{k+1}r})\nonumber\\
&=&b^{pq}(\frac 12 b^{sl})\left(-2b_{il}(\bv_j\zeta_q+\bv_q\zeta_j)+2b_{jl}(\bv_i\zeta_q+\bv_q\zeta_i)\right)+o(e^{-\frac{2n}{k+1}r})\nonumber\\
&=&-\delta_i^s(\bv_j\zeta^p+\bv^p\zeta_j)+\delta_j^s(\bv_i\zeta^p+\bv^p\zeta_i)+o(e^{-\frac{2n}{k+1}r}).
\end{eqnarray}
Finally, the desired result (\ref{aim_appen.}) follows from (\ref{I_appen.}) and (\ref{II_appen.}).
\qed

}

\part {Alexandrov-Fenchel inequalites in $\H^n$}

\setcounter{section}{5}

\section{
Introduction}

The second part of this paper is about  weighted Alexandrov-Fenchel inequalities in $\H^n$, which is used to prove the Penrose inequality
for asymptotically hyperbolic graphs in the last Section.
  This part has its own and independent interest.
For the convenience of  the reader we give an introduction on the Alexandrov-Fenchel inequalities.

The classical isoperimetric inequality
and its generalization, the Alexandrov-Fenchel inequalities,
play an important role in  integral geometry, convex  geometry and differential geometry. Let $\Omega$ be a smooth bounded domain in $\R^n$ with boundary $\Sigma$.
The classical isoperimetric inequality is
\beq\label{eq001}
Area(\Sigma)\geq n^{\frac{n-1}{n}}\omega_{n-1}^{\frac 1n}Vol(\Omega)^{\frac{n-1}{n}}.
\eeq
Equality holds if and only if $\Omega$ is a geodesic ball.
When $n=2$, \eqref{eq001} is
\beq\label{eq002} L^2\ge 4\pi A,\eeq
where $L$ is the length of a  curve $\zeta$ in $\R^2$ and $A$ is the area of the enclosed domain by $\zeta$. The Alexandrov-Fenchel inequalities
(in fact, its special class) are
\begin{equation}\label{eq003}
\int_{\Sigma}\s_kd\mu \ge C_{n-1}^k \omega_{n-1} \left(  \frac 1{C_{n-1}^j}\frac 1{\omega_{n-1}}\int_{\Sigma}\s_{j} d\mu \right)^{\frac {n-1-k}{n-1-j}}, \quad 0\le j< k\le  n-1,
\end{equation}
for any convex hypersurface  $\Sigma$. These inequalities are optimal, in the sense that equality holds if and only if $\Sigma$ is a geodesic sphere.
The Alexandrov-Fenchel inequalities  have been also  extended to certain class of non-convex hypersurfaces. See \cite{ChangWang,GuanLi3, Huisken} for instance.

It is natural to ask if the isoperimetric inequality and the Alexandrov-Fenchel inequalities  hold in the hyperbolic space.
The motivations to study this problem come  from integral geometry and also from the recent study of the Penrose inequality for various mass.
The classical isoperimetric problem between volume and area was solved by Schmidt \cite{Schmidt} 70 years ago. When $n=2$ the corresponding
isoperimetric inequality is
\[ L^2\ge 4\pi A+A^2,\]
where $L$ is the length of a  curve $\zeta$ in $\H^2$ and $A$ is the area of the enclosed domain by $\zeta$.
However, unlike the Euclidean space, for general $n$ there is no such explicit form.

There are many attempts to establish  Alexandrov-Fenchel  type inequalities in the
hyperbolic space $\H^n$. See,  for example, \cite{BM,GS,Schlenker}. In \cite{GS}, Gallego-Solanes proved by using integral geometry
the following interesting inequality for convex domains in $\H^n$,
\be \label{gs}
\int_\Sigma \s_k d\mu  > cC_{n-1}^k |\Sigma|,\ee
where $c=1$ if $k>1$ and $|\Sigma|$ is the area of $\Sigma$. Here $d\mu$ is the area element of the induced metric $\gamma$ from the hyperbolic space and $\sigma_k$ is defined in (\ref{add_sigmak}).
The above inequality \eqref{gs} ($k>1$) is sharp in the sense that the constant $c$ could not be improved.
However, this inequality is far away from being optimal, especially when $|\Sigma|$ is small. One may compare it with the optimal inequalities given below.

There are two classes of the
  Alexandrov-Fenchel inequalities: One is without  a weight $V$ and another with a weight $V$.
  The weight $V$ is an element in $\mathbb{N}_b^1$ considered in Part I. Here as before we fix it
\[
V=\cosh r,
\]
in $\H^n=\R^+\times \S^{n-1}$ with the hyperbolic metric $b=dr^2+\sinh ^2 r g_{\S^{n-1}}.$ Here $r$ is the hyperbolic distance to a fixed point $x_0$.

 The Alexandrov-Fenchel inequalities without weight are closely related to integral geometry in $\H^n$.

 \begin{theo}[\cite{LWX,GWW_AF2, GWW_AFk, WX}]
 \label{AF} Let $1\le k\le n-1$.  Any  horospherical convex hypersurface $\Sigma$ in $\H^n$
satisfies
\begin{equation}\label{eq005}
\int_{\Sigma}\s_{k}d\mu\geq C_{n-1}^k\omega_{n-1}\bigg\{\bigg(\frac{|\Sigma|}{\omega_{n-1}}\bigg)^{\frac{2}{k}}+\bigg(\frac{|\Sigma|}{\omega_{n-1}}\bigg)^{\frac{2}{k}\frac{(n-k-1)}{n-1}}\bigg\}^{\frac {k}{2}}.
\end{equation} Equality holds if and only if $\Sigma$ is a geodesic sphere.
 \end{theo}

Inequality \eqref{eq005} was called as a hyperbolic Alexandrov-Fenchel inequality in \cite{GWW_AF2}.
\eqref{eq005} was proved in \cite{LWX} for $k=2$ under a weaker condition that $\Sigma$ is star-shaped and 2-convex,
in \cite{GWW_AF2} for $k=4$ and in \cite{GWW_AFk} for general even $k$.
For $k=1$, \eqref{eq005} was proved in \cite{GWW_AFk} with a help of a result of Cheng and Xu \cite{CZ}.
For general odd integer $k$, inequality \eqref{eq005} was proved very recently in \cite{WX}.
Inequality \eqref{eq005} with odd $k$ will be used in this paper.

The proof of inequality \eqref{eq005}  with even $k$ in  \cite{LWX,GWW_AF2,GWW_AFk} uses various inverse curvature flows studied by Gerhardt \cite{Gerhardt}.
One of the crucial step is to show the monotonicity of a geometric integral under some inverse curvature flow. This geometric integral is
in fact the integral of the Gauss-Bonnet curvature $L_k(g)$ of the induced metric $g$ on the embedded hypersurface  $\Sigma\subset \H^n$
\begin{equation}\label{eq006}
\int_{\Sigma} L_k(g) d\mu (g).
\end{equation}
Hence, by using this method one in fact obtains an optimal Sobolev type inequality for (\ref{eq006}), which then implies the Alexandrov-Fenchel inequality \eqref{eq005}.
The proof of inequality \eqref{eq005}  in \cite{WX} works for all $k$. In the proof, a quermassintegral preserving curvature flow was used.
The quermassintegral will also be used in this paper. For its definition see at the end of this Section. Theorem  \ref{AF} was also proved by Guan-Li under a condition that the hypersurface is star-shaped, together with a technical condition,
by using a modified inverse curvature flow \cite{GuanLi}.

The Alexandrov-Fenchel inequalities with weight for $k=1$ was studied in \cite{BHW} and \cite{dLG}, where they were called
Minkowski type inequalities.
Motivated by the study of the quasi-local mass and the Penrose inequality, Brendle-Hung-Wang \cite{BHW}
established the following Minkowski type inequalities  (i.e., $k=1$)
\begin{equation}\label{eq03}
    \int_{\Sigma}\bigg(V\s_1-(n-1)\langle\bar{\nabla}V,\nu\rangle\bigg)d\mu\geq(n-1)\omega_{n-1}^{\frac 1{n-1}}{|\Sigma|}^{\frac{n-2}{n-1}},
\end{equation}
and  de Lima and Gir\~ao \cite{dLG}  proved the following related inequality
\begin{equation}\label{eq04}
    \int_{\Sigma}V\sigma_1 d\mu\geq (n-1)\omega_{n-1}\left(\left(\frac {|\Sigma|}{\omega_{n-1}}\right)^{\frac{n-2}{n-1}}+\left(\frac {|\Sigma|}{\omega_{n-1}}\right)^{\frac{n}{n-1}}\right),
\end{equation}
 if $\Sigma$ is star-shaped and mean convex (i.e. $\sigma_1>0$).
 The method to prove \eqref{eq03} is still the use of an inverse curvature flow and
 also  works for anti-de Sitter Schwarzschild manifolds. Moreover this method motivates the work of \cite{LWX} and \cite{GWW_AF2, GWW_AFk}.
 Inequality \eqref{eq04} was mentioned in Part I.

 It is natural to ask if general weighted Alexandrov-Fenchel inequalities hold. In this part of the paper, we give an affirmative answer, at least for
 horospherical convex hypersurfaces.

\begin{theo} \label{wAF}
Let  $\Sigma$ be a horospherical convex hypersurface in the hyperbolic space $\H^n$. We have
\begin{equation}\label{eq1_thm}
  \int_{\Sigma}V p_{2k+1} d\mu \ge \omega_{n-1}{\left(\left(\frac{|\Sigma|}{\omega_{n-1}}\right)^{\frac{n}{(k+1)(n-1)}}+\left(\frac{|\Sigma|}{\omega_{n-1}}\right)^{\frac{n-2k-2}{(k+1)(n-1)}} \right)}^{k+1}.
\end{equation}
Equality holds if and only if $\Sigma$ is a centered geodesic sphere in $\H^n$.
\end{theo}
By a centered geodesic sphere in $\H^n$ we mean $\{ r= r_0\}$, a geodesic sphere centered at the fixed point $x_0$, for some constant $r_0>0$.
Here and in the following, for the simplicity of notation we denote by
\begin{equation}\label{pj}
p_{j}=\frac1{C_{n-1}^j}\sigma_j,
\end{equation}
the normalized $j$-th mean curvature.

For the proof of this Theorem, we first need several refined Minkowski integral formulas,
which will be given in Section 7. (In fact, unlike the usual Minowski identity \eqref{Minkowski_Identity} below,
what we have are inequalities. In order to distinguish with
the Minkowski type inequalities  obtained in \cite{BHW} and \cite{dLG}, we call them  Minkowski integral formulas.)
A crucial point is to show  the following inequality
\begin{equation} \label{ineq1_1}
 E(\Sigma):=\int_{\Sigma}  Vp_{k+1} d\mu - \int_\Sigma \left(Vp_{k-1}+\frac{p_{k+1}}{V}\right)d\mu \ge 0.
\end{equation}
To show this inequality we use the following ``conformal flow"
\[ \frac {d}{dt}\Sigma(t)  =-V\nu,\]
which was used first by Brendle \cite{Brendle} to prove his generalized Heintze-Karcher inequality. With the Minkowski integral formulas given in Section 7 below,
we prove that $E$ is non-increasing along this conformal flow. Using the monotonicity of the quermassintegral  we prove that $E(\Sigma(t))$ tends to $0$, when hypersurfaces $\Sigma(t)$ shrink to a point along the conformal flow. Therefore we have $E(\Sigma)\ge 0$ for any horospherical  convex hypersurface.
Inequality \eqref{ineq1_1} enables  us to use an iteration argument as follows.
When $k=0$, \eqref{eq1_thm} is just  \eqref{eq04}, which was proved
in \cite{dLG}. Assume that  \eqref{eq1_thm}  holds for $k-1$, we then use  \eqref{ineq1_1} and
Theorem \ref{AF} to show that \eqref{eq1_thm}  holds for $k$.

It is an interesting question if Theorem \ref{wAF} holds under the weaker condition that the hypersurface is convex, or even that
the hypersurface is just so-called $k$-convex. In our proof of   Theorem \ref{wAF},  the horospherical condition is used just in the proof of Theorem \ref{thm3.2},  which we believe  is unnecessary,  and in the use of Theorem  \ref{AF}. For Theorem \ref{AF} there are evidence
that the convexity should be enough in \cite{GuanLi}.

We end this section by recalling the definition of the quermassintegrals \cite{Santos, Schneider,Solanes}.
For a (geodesically) convex domain $K\subset \H^n$ with boundary $\partial K=\Sigma$, the \textit{quermassintegrals} are defined by
\begin{eqnarray}\label{quer0}
&&W_k(K):=\frac{(n-k)\omega_{k-1}\cdots\omega_0}{n\omega_{n-2}\cdots\omega_{n-k-1}}\int_{\mathcal{L}_k}\chi(L_k\cap K)dL_k, \quad k=1,\cdots,n-1;
\end{eqnarray}
where $\mathcal{L}_k$ is the space of $k$-dimensional totally geodesic subspaces $L_k$ in $\H^n$ and $dL_k$ is the natural (invariant) measure on $\mathcal{L}_k$. The function $\chi$ is given by $\chi(K)=1$  if $K\neq \emptyset$ and $\chi(\emptyset)=0.$
For simplicity, we also use the convention
\begin{eqnarray*}
&& W_0(K)=\hbox{Vol}(K), \quad W_n(K)=\frac{\omega_{n-1}}{n}.
\end{eqnarray*}
Remark that by definition we know
\[ W_1(K)=\frac1n |\p K|.\]
 From integral geometry we know that
the quermassintegrals and the curvature integrals in $\H^n$ do not coincide. Nevertheless they  are closely related (see e.g. \cite{Solanes}, Proposition 7):
\begin{equation}\label{relation1}
\int_\Sigma p_k d \mu = n\left( W_{k+1}(K)+\frac{k}{n-k+1}W_{k-1}(K)\right),\quad k=1,\cdots,n-1,
\end{equation}

\section{Minkowski integral formulas}
Let ${u}=\langle \bar\nabla V,\nu\rangle>0$ be the support function, where $\nu$ is the outer normal vector.  Here and in the following, we denote the connections on  $\H^n$ and $\Sigma$  by $\bar\nabla$ and $\nabla$ respectively.
In the following, for a hypersurface $(\Sigma^{n-1},g)$ embedded in the hyperbolic space $\H^n$, we denote the second fundamental form by $h_{ij}$ and the shape operator $h_i^j:=h_{ik}g^{jk}$. The $k$-th Newton transformation is defined in (\ref{Newtondef}). Before stating the main results, let us collect some basic facts with the weight.
\begin{lemm}
\begin{enumerate} \item[(1)] The gradient vector field $\bar \nabla V$ of the weight function $V$ is a conformal vector field, i.e.,
 \begin{equation}\label{conformal}
 \bar\nabla_X \bar\nabla V= VX,
\end{equation}
  for any vector field $X$.
  \item [(2)] We have the following Minkowski identity with weight $V$
\begin{equation}\label{Minkowski}
\nabla_j(T_k^{ij}\nabla_i V)=-(k+1)u \sigma_{k+1}+(n-(k+1))\sigma_k V.
\end{equation}
\item [(3)] There is a relation between the weight $V$ and the support function $u$
\begin{equation}\label{fact}
V^2=1+{u}^2+|\nabla V|^2.
\end{equation}
\end{enumerate}
\end{lemm}
\begin{proof}
\eqref{conformal} is well-known.
For (\ref{Minkowski}), one can calculate it directly with the help of (\ref{conformal}) or see (8.4) in \cite{ALM}.
To prove (\ref{fact}),  first note that the  orthogonal composition $$\bar\nabla V=\nabla V+u\nu$$ implies
 $$|\bar\nabla V|^2=|\nabla V|^2+u^2.$$
Recalling $V=\cosh r$ and the simple fact $|\bar\nabla r|=1$,  we  have
$$
V^2-1=\sinh^2 r=|\bar\nabla V|^2.$$
Thus we complete the proof.
\end{proof}
In view of (\ref{pj}), it follows from \eqref{Minkowski} that we have the following well-known Minkowski integral formula between $p_k$ and $p_{k+1}$,
\begin{equation}\label{Minkowski_Identity}
\int_{\Sigma} u p_{k+1}d\mu= \int_{\Sigma}V p_k d\mu.
\end{equation}
This is the classical Minkowski integral identity in $\H^n$.

In order to prove our optimal inequalities, we need to generalize the Minkowski  type  identity between $p_k$ and $p_{k+1}$,
which are now only inequalities. See \eqref{eq2}, \eqref{eq4}, Proposition  \ref{prop1k}  and \eqref{ineq1} below.
To distinguish between such inequalities  and the Minkowski type inequality (see \eqref{eq03} and \eqref{eq04}),
we call them   {\it Minkowski  integral formulas},  between integrals involving $\sigma_k$ and $\sigma_{k+1}$.

\begin{prop}\label{prop1}
Let $\Sigma$ be a convex hypersurface in the hyperbolic space $\H^n$ and for any integer $1\leq k\leq n-1$. We have
\begin{equation}\label{eq1}
\int_{\Sigma}{u}Vp_kd\mu= \int_{\Sigma}V^2p_{k-1}d\mu+\frac{1}{C_{n-1}^k}\int_{\Sigma}\frac{1}{k}(T_{k-1})^{ij}\nabla_iV \nabla_j Vd\mu.
\end{equation}
Moreover, we have
\begin{equation}\label{eq2}
\int_{\Sigma}{u}Vp_kd\mu\geq \int_{\Sigma}V^2p_{k-1}d\mu.
\end{equation}
Equality holds if and only if $\Sigma$ is a centered geodesic sphere in $\H^n$.

\end{prop}
\begin{proof}
In view of (\ref{Minkowski}), we have
\begin{equation}\label{eq3}
\frac{1}{kC_{n-1}^k}\nabla_j(T_{k-1}^{ij}\nabla_i V)=-{u}p_k+p_{k-1}V.
\end{equation}
Multiplying the above  equation by the function $V$ and integrating by parts, we obtain the desired result (\ref{eq1}).
The convexity of $\Sigma$ implies that   $(T_{k-1})^{ij}$ is positively definite (for the proof see \cite{Guan} for instance), namely,
$$(T_{k-1})^{ij}\nabla_i V \nabla_j V\geq 0.$$
Hence (\ref{eq2}) holds. When  equality  in \eqref{eq2} holds, we have $\nabla V=0$ which implies that $\Sigma$ is a centered geodesic sphere in $\H^n.$
\end{proof}

\begin{prop}\label{prop2}
Let  $\Sigma$ be a  convex hypersurface in the hyperbolic space $\H^n$ and for any integer $1\leq k\leq n-1$. We have
\begin{equation}\label{eq4}
\int_{\Sigma}({u}^2p_k-{u}Vp_{k-1})d\mu=\frac{1}{kC_{n-1}^k}\int_{\Sigma} (T_{k-1})^{ij}\nabla_i V\nabla_j {u}d\mu \ge 0.
\end{equation}
Equality holds if and only if $\Sigma$ is a centered geodesic sphere in $\H^n$.
\end{prop}
\begin{proof}
Multiplying (\ref{eq3}) by the support function ${u}$ and integrating by parts, we have
\begin{equation}\label{eq4'}
\int_{\Sigma}({u}^2p_k-{u}Vp_{k-1})d\mu=\int_{\Sigma}\frac{1}{kC_{n-1}^k} (T_{k-1})^{ij}\nabla_i V\nabla_j {u}d\mu.
\end{equation}
Next we compute
$$
\nabla_i{u}=\bar\nabla_i {u}=\langle \bar\nabla_i\bar\nabla V, \nu\rangle + \langle \bar\nabla V,  \bar\nabla_i \nu\rangle=\langle \nabla V, \nabla_i \nu\rangle=\nabla_l V h^{l}_i,
$$
where $h_{ij}$ is the second fundamental form of $\Sigma$ in $\H^n$.
Here we have used the fact  \eqref{conformal} that the vector filed $ \bar\nabla V$ is conformal,
and $ \bar\nabla_i \nu$ has only tangential part and thus the tangential part of $ \bar\nabla V$ is $\nabla V$. Going back into (\ref{eq4'}), we obtain
$$
\int_{\Sigma} ({u}^2 p_kd\mu-{u} V p_{k-1} )d\mu=\frac{1}{C_{n-1}^k}\int_{\Sigma} \frac{1}{k}(T_{k-1})^{ij}  h^{l}_j \nabla_lV\nabla_i Vd\mu.
$$
We note that  $(T_{k-1})^{ij}$ and ${h^{li}}$ are both positive-definite and the multiplication is commutative.
 Thus, by a simple fact of linear algebra, we know that the product of matrices  $(T_{k-1})^{ij} h^{l}_i$ is still positive-definite, and hence
\begin{equation}\label{eq5}
\int_{\Sigma}(T_{k-1})^{ij} h^{l}_j\nabla_lV\nabla_iVd\mu\geq 0.
\end{equation}
 As a consequence, we have
$$
\int_{\Sigma} {u}^2 p_kd\mu\ge \int_{\Sigma} {u}V p_{k-1}d\mu.
$$
When  equality holds, we have $\nabla V\equiv 0$.  Hence, $\Sigma$ is  a centered geodesic sphere.
\end{proof}

For the later use, we need the following inequalities.

\begin{prop}\label{prop1k}
Let  $\Sigma$ be a convex hypersurface in the hyperbolic space $\H^n$ and for any integer $1\leq k\leq n-2$. We have
\begin{equation}\label{eq1_prop1k}
\ds \int_{\Sigma} V^2p_{k+1}d\mu\ge \int_{\Sigma}V^2p_{k-1}d\mu+\int_{\Sigma} p_{k+1}d\mu,
\end{equation}
and
\begin{equation}\label{eq2_prop1k}
\ds \int_{\Sigma} uVp_{k+1}d\mu\ge \int_{\Sigma} uVp_{k-1} d\mu+\int_{\Sigma} \frac{u}{V}p_{k+1}d\mu.
\end{equation}
Equality in the above inequalities holds if and only if $\Sigma$ is a centered geodesic sphere in $\H^n$.
\end{prop}
\begin{proof}
Applying (\ref{fact}), we obtain
$$
\int_{\Sigma} V^2p_{k+1}d\mu= \int_{\Sigma} u^2p_{k+1}d\mu+  \int_{\Sigma} p_{k+1}d\mu+ \int_{\Sigma} |\nabla V|^2 p_{k+1}d\mu.
$$
 From Proposition \ref{prop1} and Proposition \ref{prop2},
\[\begin{array}{rcl}
\ds
\int_{\Sigma} u^2p_{k+1}d\mu \geq  \ds   \int_{\Sigma} uV p_{k}d\mu\geq  \ds   \int_{\Sigma} V^2 p_{k-1}d\mu.
\end{array}\]
Therefore, the desired inequality (\ref{eq1_prop1k}) follows. To prove (\ref{eq2_prop1k}), using (\ref{fact}) and Proposition \ref{prop1}, Proposition \ref{prop2} again, we have
\[\begin{array}{rcl}
\ds
\int_{\Sigma} uV p_{k+1} d\mu\ge \int_{\Sigma} V^2 p_{k}=  \int_{\Sigma} p_k d\mu +   \int_{\Sigma} u^2 p_{k}d\mu+ \int_{\Sigma} |\nabla V|^2 p_{k}d\mu\ge  \int_{\Sigma} p_k d\mu +   \int_{\Sigma} uV p_{k-1}d\mu .
\end{array}\]
Multiplying (\ref{eq3}) by the function $\frac{1}{V}$ and integrating by parts, we have
\begin{eqnarray}\label{eq2k}
\int_{\Sigma}    p_{k}d\mu&=&\int_{\Sigma}\frac{u}{V} p_{k+1}d\mu+\frac{1}{(k+1)C_{n-1}^{k+1}}\int_{\Sigma} \frac{1}{V^2}(T_{k})^{ij}\nabla_iV\nabla_j Vd\mu\nonumber\\
&\ge&\int_{\Sigma}\frac{u}{V} p_{k+1}d\mu.
\end{eqnarray}
Here in the last inequality, we have used the fact that $(T_{k})^{ij}$ is positive-definite.
Hence we prove (\ref{eq2_prop1k}).
The equality cases follow readily.
\end{proof}

\section{A crucial  Minkowski integral formula }

In this section, we consider the following functional
\begin{equation}\label{generalE}
E:=\int_{\Sigma}\bigg(Vp_{k+1}-Vp_{k-1}-\frac{p_{k+1}}{V}\bigg)d\mu.
\end{equation}
Before discussing further, let us recall some basic facts of the general evolution equations.
Precisely, consider a one-parameter family $X(t,.):\Sigma^{n-1}\rightarrow \mathbb{H}^n,\, t\in[0,\epsilon)$ of closed, isometrically embedded hypersurfaces evolving by
\begin{equation}\label{flow}
\frac{\partial X}{\partial t}=F\nu,
\end{equation}
where $\nu$ is the outward unit normal to $\Sigma_t=X(t,.)$ and $F$ is a general speed function. For the convenience of the reader, we collect some evolution formulas in the following lemma.
\begin{lemm}\label{evo.lemm1}
Along  flow (\ref{flow}), we have
\begin{enumerate}[(1)]
\item $\frac{\partial}{\partial t}d\mu=F\sigma_1d\mu,$
\vspace{2mm}
\item $\frac{\partial V}{\partial t}={u} F,$
\vspace{2mm}
\item $\frac{\partial\sigma_k}{\partial t}=-T_{k-1}^{ij}\nabla_i\nabla_j F-F(\sigma_1\sigma_k-(k+1)\sigma_{k+1})+(n-k)F\sigma_{k-1},$
\vspace{2mm}
\item $\ds\frac{\partial}{\partial t}\int_{\Sigma}Vp_k d\mu=\int_{\Sigma}\left((k+1){u} p_k +(n-k-1)V p_{k+1}\right)Fd\mu,$
\vspace{3mm}
\item  For $l\geq 0$, we have \[
\begin{array}{rcl} \ds\vs
\frac{\partial}{\partial t}\int_{\Sigma}\frac{p_k}{V^l} d\mu& =&\!\! -\ds l(k\!+\!1) \int_\Sigma  \frac {p_k}{V^{l+1}} uFd\mu- \frac{l(l\!+\!1)}{C_{n-1}^k} \int_\Sigma
\frac{(\nabla_iV)(\nabla_j V)}{V^{l+2}} T^{ij}_{k-1} Fd\mu\\
&& \ds +k(l\!+\!1)  \int_\Sigma  \frac  {p_{k-1} } {V^{l}}Fd\mu+ \int _{\Sigma} (n-k-1)\frac {p_{k+1}} {V^{l}}Fd\mu.
\end{array}\]
In particular, under flow $F=-V$ we have the following simple form
$$\frac{\partial}{\partial t}\int_{\Sigma}\frac{p_k}{V^l}d\mu=\int_{\Sigma}\frac{(l-k)u p_{k}-(n-k-1)Vp_{k+1}}{V^l}d\mu.$$
\end{enumerate}
\end{lemm}
\begin{proof}
(1) and (2) are included in \cite{Huisken1}, and (3) follows from \cite{Reilly}.
Here we provide a proof for (4).
In view of (\ref{Minkowski}), we have along flow (\ref{flow}) that
\begin{eqnarray}
\frac{\partial}{\partial t}\int_{\Sigma}V\sigma_k d\mu&=&\int_{\Sigma}\frac{\partial V}{\partial t}\sigma_k d\mu+\int_{\Sigma}V\frac{\partial\sigma_k}{\partial t} d\mu+\int_{\Sigma}V\sigma_k (F\sigma_1)d\mu\nonumber\\
&=&\int_{\Sigma}\bigg\{uF\sigma_k\!+\!V\big(\!-\!T_{k-1}^{ij}\nabla_i\nabla_j F\!-\!F(\sigma_1\sigma_k\!-\!(k\!+\!1)\sigma_{k+1})\!+\!(n\!-\!k)F\sigma_{k\!-\!1}+\!F\sigma_1\sigma_k\big) \! \bigg\}d\mu\nonumber\\
&=&\int_{\Sigma}\bigg(uF\sigma_k-V(T_{k-1}^{ij}\nabla_i\nabla_j F-F(n-k)\sigma_{k-1})+(k+1)V\sigma_{k+1}F\bigg)d\mu\nonumber\\
&=&\int_{\Sigma}\bigg(uF\sigma_k-\left(\nabla_j( T_{k-1}^{ij}\nabla_i V) -(n-k)\sigma_{k-1}V-(k+1)V\sigma_{k+1}\right) F \bigg)d\mu\nonumber\\
&=&\int_{\Sigma}(k+1)(u\sigma_k +V\sigma_{k+1})Fd\mu,\nonumber
\end{eqnarray}
here we used the fact $T_{k-1}$ is divergence-free and in the fifth equality we used (\ref{Minkowski}). Thus
\begin{equation}\label{Vpk}
\frac{\partial}{\partial t}\int_{\Sigma}Vp_k d\mu=\int_{\Sigma}\bigg((k+1){u} p_k +(n-k-1)V p_{k+1}\bigg)Fd\mu.
\end{equation}
The proof of (5) follows from a similar computation.
\end{proof}

In order to show that $E$ is non-negative,
we use the following flow
\begin{equation}\label{Vflow}
\frac{\partial X}{\partial t}=-V\nu,
\end{equation}
which was first used in \cite{Brendle}. 

\begin{theo} \label{thm3.2} Let $1\le k< n-1$.   Any   horospherical convex hypersurface  in the hyperbolic space $\H^n$ satisfies
\begin{equation}
\label{ineq1}
\int_{\Sigma} Vp_{k+1} \ge \int_\Sigma \left(Vp_{k-1}+\frac{p_{k+1}}{V}\right)d\mu.
\end{equation}
\end{theo}

By Lemma \ref{evo.lemm1}, one immediately obtains the evolution equation of $E$ along flow (\ref{Vflow}),
\begin{eqnarray}\label{E'}
\frac{dE}{dt}
&=&\int_{\Sigma}\bigg\{\big((k+2)u p_{k+1}+(n-k-2)Vp_{k+2}\big)(-V)-\big(kup_{k-1}+(n-k)Vp_{k}\big)(-V)\bigg\}d\mu\nonumber\\
&&+\int_{\Sigma}\bigg(k\frac{u p_{k+1}}{V}+(n-k-2)p_{k+2}\bigg)d\mu\nonumber\\
&=&-(n-k-2)\int_{\Sigma}\left(V^2p_{k+2}-V^2p_{k}-p_{k+2}\right)d\mu-2\int_{\Sigma}\left(u Vp_{k+1}-V^2 p_{k}\right)d\mu\nonumber\\
&&-k\int_{\Sigma}\left(u Vp_{k+1}-u Vp_{k-1}-\frac{up_{k+1}}{V}\right)d\mu.
\end{eqnarray}
It follows, together with (\ref{eq2}) and Proposition \ref{prop1k} that the monotonicity of the functional $E$
$$\frac{dE}{dt}\leq 0,$$
along the flow (\ref{Vflow}). That is, we have showed that

\begin{prop} \label{pro3.1} The functional $E$ defined in (\ref{generalE}) is non-increasing under  flow (\ref{Vflow}).
\end {prop}

For the proof Theorem \ref{thm3.2}, we need the following two more lemmas.

\begin{lemm}\label{lem_b1} Flow (\ref{Vflow}) preserves the horospherical convexity.
\end{lemm}

\begin{proof}
A direct computation gives (or see \cite{Huisken1} for instance)
$$
\p_t h^i_{j}=\nabla^i\nabla_j V+V((h^2)^{i}_j-\delta_{j}^i).
$$
Set
$$
\tilde h^i_{j}:=h^i_j-\delta^i_j,
$$
Noting that the fact $$\nabla^i\nabla_j V=V\delta^{i}_j-uh^i_j,$$
which follows from (\ref{conformal}),
one has
$$
\p_t \tilde h^i_j=(V-u)\delta^i_j+V((\tilde h^2)^i_{j}+2(\tilde h)^i_j)-u\tilde h^i_j.
$$
Let $a$ be a unit vector such that
$$
\tilde h^i_j a^j=0.
$$
Then we have
$$
\left((V-u)\delta^i_j+V((\tilde h^2)^i_j+2(\tilde h)^i_j)-u\tilde h^i_j\right)a_i a^j=V-u>0.
$$
The Lemma  follows easily.\end{proof}

\begin{lemm} The quermassintegrals are monotone under the set inclusion, i.e.
\begin{equation}\label{eq_z01} W_k(K_1) \le W_k(K_2), \quad\hbox{ if } K_1 \subset K_2.\end{equation}
As a consequence,
\begin{equation}\label{eq_z02}
\int_{\partial K_1} \sigma_k d \mu \le \int_{\partial K_2} \sigma_k d \mu, \quad\hbox{ if } K_1 \subset K_2.\end{equation}
\end{lemm}

\begin{proof} \eqref{eq_z01} follows easily from the definition of the quermassintegral, see \eqref{quer0}.
\eqref{eq_z02} follows from \eqref{eq_z01} and \eqref{relation1}
\end{proof}

\

\noindent{\it Proof Theorem \ref{thm3.2}}.
For any horospherical convex hypersurface $\Sigma$ we consider the flow with $F=-V$. Let $T^*\in (0,\infty)$ be the maximal time of existence of the flow.
It is clear that $T^*$ is finite. Let $\Sigma_t$ be the evolved hypersurface for $t\in [0, T^*)$ and define $r(t)$ and $R(t)$ the inner radius and outer
radius of $\Sigma_t$ respectively. It is clear that $r(t)\to 0$ when $t\to T^*$. By Lemma \ref{lem_b1}, we know that every $\Sigma_t$ is horospherical convex. As a  feature of  horospherical convex hypersurface, $R(t)$ is controlled by $3\sqrt{r(t)}$ from above when $r(t)$ is sufficiently small, and hence we have that $R(t)$ also converges to $0$, as $t\to T^*$. See \cite{BM}.
 From the  monotonicity \eqref{eq_z02}, we have
 \[ \int_{\Sigma_t} \sigma_k d \mu \le \int_{\partial B_{R(t)}} \sigma_k d \mu,\]
 where $ B_{R(t)}$ and $ \partial B_{R(t)}$  are  the geodesic ball, the geodesic sphere   of radius $R(t)$ respectively,  for $ \Sigma_t\subset B_{R(t)}$.
It is easy to check that
 \[ \int_{\Sigma_t} \sigma_k d \mu\le   \int_{\partial B_{R(t)}} \sigma_k d \mu \to 0, \quad \hbox{ as } R(t) \to 0.\]
 It follows readily that
\[\int_{\Sigma_t} V^l \s_k d \mu \to 0,\]
which implies that $E(\Sigma_t) \to 0$. By Proposition \ref{pro3.1} we have $E(\Sigma)\ge 0$.
\qed

\

We remark that here we have used the horospherical convexity. We believe that in the argument the convexity should be enough,
if we use a more precise information about the conformal flow \eqref{Vflow}.

\section{Weighted Alexandrov-Fenchel inequalities}

Now we begin to prove Theorem \ref{wAF}.

\

\noindent{\it Proof of Theorem \ref{wAF}.}
Due to Theorem \ref{thm3.2},  we are able to use  the induction argument to prove this theorem. When $k=0$, (\ref{eq1_thm}) is just (\ref{eq04}), which was proved  in \cite{dLG}. Assume (\ref{eq1_thm}) holds for $k-1$, namely the following holds,
\begin{equation}\label{eq2_thm}
 \int_{\Sigma} V p_{2k-1}  d\mu \ge \omega_{n-1}{\left\{\left(\frac{|\Sigma|}{\omega_{n-1}}\right)^{\frac{n}{k(n-1)}}+\left(\frac{|\Sigma|}{\omega_{n-1}}\right)^{\frac{n-2k}{k(n-1)}} \right\}}^{k}.
\end{equation}
We need to show that (\ref{eq1_thm}) holds for $k$. For the simplicity of notation
we denote
\[\|\Sigma\|=\frac{|\Sigma|}{\omega_{n-1}}.
\]
First recall (\ref{eq005}) that
$$\int_{\Sigma} p_{2k+1}d\mu\ge   \omega_{n-1}\left({\|\Sigma\|}^{\frac{2}{2k+1}}+{\|\Sigma\|}^{\frac{2(n-2k-2)}{(2k+1)(n-1)}}\right)^{\frac{2k+1}{2}}.$$
It follows, together with the H\"{o}lder inequality, that
\begin{eqnarray*}
\left(\int_{\Sigma}Vp_{2k+1}d\mu\right)\left(\int_{\Sigma}\frac{p_{2k+1}}{V}d\mu\right)&\geq&
\left(  \int_{\Sigma} p_{2k+1}d\mu\right)^2\\
&\ge&
 \omega_{n-1}^2\left({\|\Sigma\|}^{\frac{2}{2k+1}}+{\|\Sigma\|}^{\frac{2(n-2k-2)}{(2k+1)(n-1)}}\right)^{2k+1}\\
&=&\omega_{n-1}^2 {\|\Sigma\|}^2 \left(1+{\|\Sigma\|}^{-\frac{2}{n-1}}\right)^{2k+1}.
\end{eqnarray*}
Set
$$
\alpha:=\omega_{n-1}^2 {\|\Sigma\|}^2 \left(1+{\|\Sigma\|}^{-\frac{2}{n-1}}\right)^{2k+1}.$$
 From above we have
\be\label{eq_z1} \begin{array}{rcl}
\ds  \int_{\Sigma}V p_{2k+1} d\mu -\int_{\Sigma} \frac {p_{2k+1}}{V} d\mu\le \int_{\Sigma}V p_{2k+1}d\mu-\frac{\alpha}{\int_{\Sigma}V p_{2k+1} d\mu}
\end{array}.\ee
 From our crucial Minkowski integral formula \eqref{ineq1} and the induction assumption (\ref{eq2_thm}), we have
\be\label{eq_z2}\begin{array}{rcl}
\ds\vs \int_{\Sigma}V p_{2k+1} d\mu -\int_{\Sigma} \frac {p_{2k+1}}{V} d\mu&\geq& \ds \int_{\Sigma}Vp_{2k-1}d\mu\\
&\geq& \ds \omega_{n-1}\|\Sigma\|^{\frac{n}{n-1}}\left(1+\|\Sigma\|^{-\frac{2}{n-1}}\right)^k.
\end{array}\ee
We introduce an auxiliary function $f(t):=t-\frac {\alpha} t$.  Then \eqref{eq_z1} and \eqref{eq_z2} imply
\beq\label{add_zz}
f\left(\int_{\Sigma}V p_{2k+1} d\mu\right)\geq \omega_{n-1}\|\Sigma\|^{\frac{n}{n-1}}\left(1+\|\Sigma\|^{-\frac{2}{n-1}}\right)^k.
\eeq
On the other hand, one can easily check that
\begin{eqnarray*}
&&f\left( \omega_{n-1}\left({\|\Sigma\|}^{\frac{n}{(k+1)(n-1)}}+{\|\Sigma\|}^{\frac{n-2k-2}{(k+1)(n-1)}} \right)^{k+1}\right)\\
&=&\omega_{n-1} {\|\Sigma\|}^{\frac{n}{n-1}}\left(1+{\|\Sigma\|}^{-\frac{2}{n-1}}\right)^{k+1}-
\omega_{n-1}{\|\Sigma\|}^{\frac{n-2}{n-1}} \left(1+{\|\Sigma\|}^{-\frac{2}{n-1}}\right)^{k}\\
&=&\omega_{n-1} {\|\Sigma\|}^{\frac{n}{n-1}}\left(1+{\|\Sigma\|}^{-\frac{2}{n-1}}\right)^{k}\\
&\le & f\left(\int_{\Sigma}V p_{2k+1} d\mu\right) .
\end{eqnarray*}
The last inequality follows from \eqref{add_zz}.
Since $f$ is increasing on the interval $[0,+\infty)$, we have
\[\int_{\Sigma}V p_{2k+1} d\mu\ge \omega_{n-1}\left({\|\Sigma\|}^{\frac{n}{(k+1)(n-1)}}+{\|\Sigma\|}^{\frac{n-2k-2}{(k+1)(n-1)}} \right)^{k+1},
\]
the desired result (\ref{eq1_thm}). When equality holds, it follows from the equality in H\"older inequality that $V$  is constant on $\Sigma$, which yields $\Sigma$ is a centered geodesic sphere.  Hence we complete the proof.
\qed

\

We end  this section with a conjecture of ``weighted'' Alexandrov-Fenchel inequalities for even $k$.
\begin{conj} Let $k\le n-1 $ be even.
Any horospherical convex hypersurface $\Sigma$  in the hyperbolic space $\H^n$  satisfies
\begin{equation}\label{eq_conj.}
  \int_{\Sigma} V\s_{k}  d\mu \ge{C_{n-1}^{k}} \omega_{n-1}{\left(\left(\frac{|\Sigma|}{\omega_{n-1}}\right)^{\frac{2n}{(k+1)(n-1)}}+\left(\frac{|\Sigma|}{\omega_{n-1}}\right)^{\frac{2(n-k-1)}{(k+1)(n-1)}} \right)}^{\frac{k+1}2}.
\end{equation}
Equality holds if and only if $\Sigma$ is a centered geodesic sphere in $\H^n$.
\end{conj}
We have proved inequality \eqref{eq_conj.} for odd $k$. For even $k$ our argument presented in this paper still works if the induction argument could start,
i.e., if
the following inequality
\be\label{con_last}
\int_{\Sigma} V  d\mu \ge \omega_{n-1}{\left(\left(\frac{|\Sigma|}{\omega_{n-1}}\right)^{\frac{2n}{n-1}}+\left(\frac{|\Sigma|}{\omega_{n-1}}\right)^2 \right)}^{\frac{1}2},
\ee
holds. We believe that \eqref{con_last} is true. However, we could not prove it yet.
Instead, we  have a weaker version of (\ref{eq_conj.}) for even $k$, which is also optimal.

\begin{theo}
Any horospherical convex hypersurface $\Sigma$  in the hyperbolic space $\H^n$  satisfies
\begin{equation}\label{eq9.6}
  \int_{\Sigma} V\s_{2k}  d\mu \ge{C_{n-1}^{2k}} \omega_{n-1}{\left(\left(\frac{\int_{\Sigma}ud\mu}{\omega_{n-1}}\right)^{\frac{2}{2k+1}}+\left(\frac{\int_{\Sigma}ud\mu}{\omega_{n-1}}\right)^{\frac{2(n-2k-1)}{(2k+1)n}} \right)}^{\frac{2k+1}2}.
\end{equation}
Here ${u}=\langle \bar\nabla V,\nu\rangle>0$ is the support function, where $\nu$ is the outer normal vector.
Equality holds if and only if $\Sigma$ is a centered geodesic sphere in $\H^n$.
\end{theo}
\begin{proof}
As in the proof of Theorem \ref{wAF}, we adopt the induction argument.
First when $k=0$, it follows from  (\ref{fact}) and H\"older's inequality that
\begin{eqnarray*}
\ds  \left(\int_{\Sigma} V  d\mu\right)^2
&\ge &\ds \left(\int_{\Sigma} V  d\mu\right) \left(\int_{\Sigma}\frac1 V  d\mu+ \int_{\Sigma}\frac{u^2}{ V } d\mu\right) \\
&\ge & \ds \left(\int_{\Sigma} 1\;  d\mu\right)^2+  \left(\int_{\Sigma}u d\mu\right)^2.
\end{eqnarray*}
Using the fact
\begin{equation}\label{fact 2}
\left(\frac{|\Sigma|}{\omega_{n-1}}\right)^{\frac{n}{n-1}}\ge \frac{\ds\int_{\Sigma}ud\mu}{\omega_{n-1}},
\end{equation}
which has been proved in \cite{dLG} (see Proposition 3.3), we obtain the desired result for $k=0$ that
$$\ds  \left(\int_{\Sigma} V  d\mu\right)^2\ge  \ds  \omega_{n-1}^2 \left(\left(\frac{\int_{\Sigma}ud\mu}{\omega_{n-1}}\right)^2+\left(\frac{\int_{\Sigma}ud\mu}{\omega_{n-1}}\right)^{\frac{2(n-1)}{n}} \right).$$
Now we can start the induction argument.  Noting that by (\ref{eq005}) and (\ref{fact 2}), we have
\begin{eqnarray*}
\int_{\Sigma}\s_{k}d\mu&\geq& C_{n-1}^k\omega_{n-1}\bigg\{\bigg(\frac{\int_{\Sigma}ud\mu}{\omega_{n-1}}\bigg)^{\frac{2(n-1)}{k n}}+\bigg(\frac{\int_{\Sigma}ud\mu}{\omega_{n-1}}\bigg)^{\frac{2(n-k-1)}{k n}}\bigg\}^{\frac {k}{2}}\\
&=&C_{n-1}^k\omega_{n-1}\bigg(\frac{\int_{\Sigma}ud\mu}{\omega_{n-1}}\bigg)^{\frac{n-1}{n}}\bigg\{1+\bigg(\frac{\int_{\Sigma}ud\mu}{\omega_{n-1}}\bigg)^{-\frac{2}{n}}\bigg\}^{\frac {k}{2}}.
\end{eqnarray*}
The rest of the proof is essentially the same  as in the one of Theorem \ref{wAF} and we skip it.
\end{proof}

\end{document}